\documentclass[onefignum,onetabnum]{siamart190516}

\usepackage{lipsum}
\usepackage{amsfonts}
\usepackage{graphicx}
\usepackage{epstopdf}
\usepackage{algorithmic}
\ifpdf
  \DeclareGraphicsExtensions{.eps,.pdf,.png,.jpg}
\else
  \DeclareGraphicsExtensions{.eps}
\fi

\newsiamremark{remark}{Remark}
\newsiamremark{hypothesis}{Hypothesis}
\crefname{hypothesis}{Hypothesis}{Hypotheses}
\newsiamthm{claim}{Claim}

\headers{Nonuniform approximations of semilinear diffusion-wave equations}{P. Lyu and S. Vong}

\title{A symmetric fractional-order reduction method for direct nonuniform approximations of semilinear diffusion-wave equations\thanks{\funding{This work was partially supported by the Fundamental Research Funds for the Central Universities (JBK2102010), the National Natural Science Foundation of China (12071373), The Science and Technology Development Fund, Macau SAR (File no. 0005/2019/A)  and the grant MYRG2018-00047-FST from University of Macau.}}}

\author{Pin Lyu\thanks{School of Economic Mathematics, Southwestern University of Finance and Economics, Chengdu, China.
  (\email{plyu@swufe.edu.cn}).}
\and Seakweng Vong\thanks{Corresponding author. Department of Mathematics, University of Macau, Macao, China.
  (\email{swvong@um.edu.mo}).}}

\usepackage{amsopn}

\ifpdf
\hypersetup{
  pdftitle={Nonuniform approximations of semilinear diffusion-wave equations},
  pdfauthor={P. Lyu and S. Vong}
}
\fi

\usepackage{subfigure}
\usepackage{cite}
\newcommand{\zd}{\,\mathrm{d}}
\newtheorem{example}{Example}

\begin{document}

 \maketitle\normalsize


\begin{abstract}
We introduce a symmetric fractional-order reduction (SFOR) method to construct numerical algorithms on general nonuniform temporal meshes for semilinear fractional diffusion-wave equations. By using the novel order reduction method, the governing problem is transformed to an equivalent coupled system, where the explicit orders of time-fractional derivatives involved are all $\alpha/2$ $(1<\alpha<2)$. The linearized L1 scheme and Alikhanov scheme are then proposed on general time meshes. Under some reasonable regularity assumptions and weak restrictions on meshes, the optimal convergence is derived for the two kinds of difference schemes by $H^2$ energy method. An adaptive time stepping strategy which based on the (fast linearized) L1 and Alikhanov algorithms is designed for the semilinear diffusion-wave equations. Numerical examples are provided to confirm the accuracy and efficiency of proposed algorithms.
\end{abstract}



\section{Introduction}

In this paper, we consider numerical methods of the semilinear diffusion-wave equation:
\begin{align}\label{eq1}
 {\cal D}_t^{\alpha} u=\nu^2\Delta u+f(u,{\bf x},t),\quad {\bf x}\in\Omega,~t\in(0,T],
\end{align}
subject to the initial conditions $u({\bf x},0)=\varphi({\bf x})$ and $u_t({\bf x},0)=\tilde\varphi({\bf x})$ for ${\bf x}\in\Omega$, and the homogeneous boundary condition $u({\bf x},t)=0$ for ${\bf x}\in\partial\Omega$; where $\Omega=(x_l,x_r)\times(y_l,y_r)$, $1<\alpha<2$, $\nu$ is a constant, and ${\cal D}_t^{\delta}$ denotes the Caputo derivative of order $\delta$:
$${\cal D}_t^{\delta}u(t):=({\cal I}^{n-\delta} u^{(n)})(t)   \quad \mbox{for}~t>0~\mbox{and}~ n-1<\delta<n, $$
in which ${\cal I}^\beta$ represents  the Riemann-Liouville fractional integral of order $\beta$:
 $${\cal I}^\beta u(t):=\int_0^t\omega_\beta(t-s)u(s)\zd s
 \quad \mbox{with}\quad \omega_\beta(t)=\frac{t^{\beta-1}}{\Gamma(\beta)}.$$

 The diffusion-wave equation, which is also called the time-fractional wave equation, can be applied to describe evolution processes intermediate between diffusion and wave propagation. For example, it governs the  propagation of mechanical waves in viscoelastic media \cite{Mainardi2010,Mainardi2001}. The  practical applications of equation \eqref{eq1} span diversely many  disciplines, such as the  image processing \cite{ZhangW2014,Cuesta2012SP}, the universal electromagnetic, acoustic and mechanical response \cite{Nigmatullin}.

It is well known that the solutions of the sub-diffusion equations (also called the time-fractional diffusion equations) typically exhibit weak initial singularities \cite{JinIMA,Stynes-SIAM2017,Saka-Yama-moto}, and it causes that the traditional time-stepping methods fail to preserve their desired convergence rate \cite{JinIMA}. The same phenomenon occurs for the diffusion-wave equations. For example, Jin, Lazarov and Zhou \cite[Theorem A.4]{JinSIAMJSC2016} show that the solution of the linear diffusion-wave equation ($f=f({\bf x},t)$) satisfies that $\|\partial_t^m u \|_{L^2(\Omega)}\leq C_Tt^{\alpha-m}\|f\|_{W^{m-1,\infty}(0,T;L^2(\Omega))}$, $m=1,2$, if $f\in W^{1,\infty}(0,T;L^2(\Omega))$ and $\varphi={\tilde \varphi}=0$. Other studies on regularities can be found in \cite{JinSIAMJSC2016,MustaphaNM2007,Saka-Yama-moto}.
Recently some excellent works have been done on the numerical approximation of linear diffusion-wave equations taking the weak initial singularities into account. The convolution quadrature methods generated by backward difference formulas are rigorously discussed in \cite{JinSIAMJSC2016}, where the first- and second-order temporal convergence rates are obtained under proper assumptions of the given data, and their discrete maximal regularities are further studied by Jin, Li and Zhou  \cite{Jin-NM2018}.
Lately, for the problem with nonsmooth data, a Petrov-Galerkin method and a time-stepping discontinuous Galerkin method
are proposed
in
\cite{LL-Xie} (Luo, Li and Xie) and
\cite{LiWangXie2019JSC} (Li, Wang and Xie), where the temporal convergence rate is $(3-\alpha)/2$-order and about first-order respectively.
Numerical schemes with classical L1 approximation in time and the standard P1-element in space are also implemented in \cite{LiWangXie2019} to have the temporal accuracies of
${\cal O}(\tau^{3-\alpha})$ and ${\cal O}(\tau^2)$ provided the ratio $\tau^\alpha/h^2_{\min}$ is uniformly bounded.
We note that the numerical methods in the above works \cite{JinSIAMJSC2016,Jin-NM2018,LL-Xie,LiWangXie2019JSC,LiWangXie2019} are implemented on uniform temporal steps.
On the other hand, Mustapha \& McLean \cite{MustaphaSIAM2013} and Mustapha \& Sch\"{o}tzau  \cite{MustaphaIMA2014} considered the time-stepping discontinuous Galerkin methods on nonuniform temporal meshes to solve the following kind of fractional wave equation:
\begin{align}\label{integro-differential}
u_t+{\cal I}^\beta Au(t)=f(t),\quad  \mbox{for} \quad \beta\in(0,1)\quad  \mbox{and}\quad  t\in(0,T],
\end{align}
 where $A$ is a self-adjoint linear elliptic spatial operator. It can be observed that the above integro-differential problem is (mathematically) equivalent to the linear case of \eqref{eq1} under suitable assumptions on $f$ and initial data. Their methods are illuminating and efficient with good temporal accuracies.  Laplace transform methods and convolution quadrature methods on uniform temporal steps are also discussed respectively by McLean \& Thom$\acute{\mbox{e}}$e \cite{McLeanThomee2004,McLeanThomee2010} and Cuesta et al. \cite{Cuesta2006MC,Cuesta2003ANM,Cuesta2003SIAM} for the above integro-differential problem, where the reference \cite{Cuesta2006MC} is for the semilinear case $f(t)=f(u,\nabla u,x,t)$. However, the above numerical methods for solving \eqref{integro-differential} may not be easily  extended to the semilinear problem \eqref{eq1} due to the nonlinearity ${\cal D}_t^\beta f(u,t)$.

To the best of our knowledge, there are still challenges for  numerical methods of the diffusion-wave equation. In this paper, we will address  the following issues:  {\bf (i)} establishing and analyzing difference schemes by the classical L1 \cite{Oldham1974} and Alikhanov \cite{AAA} approximations on nonuniform temporal meshes (especially on more general meshes) for the semilinear diffusion-wave equation with typical weak singular solutions; {\bf (ii)} studying efficient numerical algorithms, such as the adaptive time-stepping algorithm, for the semilinear diffusion-wave equation in order to deal with the highly oscillatory variations in time since the problem \eqref{eq1} leads to a mixed behavior of diffusion and wave propagation.

Before introducing our main approach, we review two classical and popular algorithms. The first one is the L1 algorithm \cite{Oldham1974}, which was generated by Lagrange linear interpolation formula, it is a direct and convenient approximation formula in constructing numerical methods for sub-diffusion problems, e.g., \cite{SunWu2006,LinXu2007,YanSIAM2018} where it was employed on uniform temporal grids. Recently, the L1 method on graded temporal meshes, with monotonically increasing step sizes, was analyzed in Stynes, O'Riordan \& Gracia \cite{Stynes-SIAM2017} and Kopteva \cite{KoptevaMC2019} to resolve the sub-diffusion equations with weakly singular solutions. The other one is the Alikhanov algorithm, which was firstly proposed by Alikhanov \cite{AAA} by combining linear and quadratic interpolations skillfully  at an off-set time point on uniform mesh for the sub-diffusion problem with sufficiently smooth solution. Implementation of this algorithm on graded mesh was discussed by Chen and Stynes \cite{ChenStynesJSC2019} and  the second-order convergence concerning with the weak initial singularities was established. Particularly, Liao, Li and Zhang \cite{LiaoL1} presented a novel and technical framework to derive the optimal convergence result of the nonuniform L1 scheme. The techniques were then generalized in Liao, McLean and Zhang \cite{LiaoGronwall}, which were further extended to a linearized scheme for the semilinear sub-diffusion equations \cite{LiaoYanZhang2018} and the Alikhanov scheme on more general nonuniform meshes \cite{LiaoSecondOrder}. We remark that the above methods \cite{Stynes-SIAM2017,KoptevaMC2019,ChenStynesJSC2019,LiaoL1,LiaoGronwall,LiaoYanZhang2018,LiaoSecondOrder} on nonuniform meshes are all for sub-diffusion problems.

In view of the high efficiency and broad potential applications of the L1 and the  Alikhanov algorithms, it is of high scientific value to consider their nonuniform versions for resolving (at least) the weak initial singularities of the diffusion-wave problem. In \cite{SunWu2006}, Sun and Wu utilized a fixed-order reduction method, i.e., by taking an auxiliary function $v=u_t$, to rewrite a linear case of equation \eqref{eq1} to the following coupled equations:
 \begin{align}\label{standard-OR1}
 & {\cal D}_t^{\alpha-1} v=\nu^2\Delta u+f({\bf x},t),\\\label{standard-OR2}
& v=u_t,
 \end{align}
 for ${\bf x}\in\Omega,~t\in(0,T]$.
 We note that the time-fractional derivative on the auxiliary function $v$ in equation \eqref{standard-OR1} is of order $\alpha-1$ which belongs to $(0,1)$, so the system \eqref{standard-OR1}--\eqref{standard-OR2} is not structure consistency in the time derivative order point of view. The diffusion-wave equation can be solved following the standard  framework of the L1 method on uniform temporal meshes.  Although it is easy to extend the above order reduction method to the corresponding nonuniform L2 scheme (see \cite{SunWu2006} for its uniform version), we find that it may be difficult to establish its stability and convergence on more general time meshes. Therefore, for the first time, we present a new order reduction method by introducing a novel auxiliary function
 $${\bf v}= {\cal D}_t^{\frac{\alpha}{2}} {\bf u},$$
 where ${\bf u}=u-t{\tilde \varphi}$, which is a non-fixed-order reduction technique, and we call the symmetric fractional-order reduction (SFOR) method.
 The semilinear diffusion-wave equation \eqref{eq1} is then skillfully rewritten  as coupled equations with nice structure or having the feature of structure consistency, i.e. \eqref{eq-new1}--\eqref{eq-new2}, see  Section \ref{SFOR} for more details.
 Basing on this equivalent formulation, we can construct the implicit and linearized L1 and Alikhanov algorithms on possible nonuniform time partitions $0=t_0<t_1<\cdots<t_N=T$ for a given positive integer $N$, and discuss their unconditional convergence by utilizing the framework of \cite{LiaoL1,LiaoGronwall,LiaoSecondOrder,LiaoYanZhang2018}.
 Throughout this paper we assume that the solution satisfies the following regularity:
  \begin{align}\label{regularity}
 &\|\partial_t^{(k)}u\|_{H^4(\Omega)}\leq C_u(1+t^{\sigma_1-k}) \quad \mbox{and}\quad  \|\partial_t^{(k)}{\bf v}\|_{H^4(\Omega)}\leq C_u(1+t^{\sigma_2-k}) ,\quad k=0,1,2,3,
 \end{align}
 for $t\in(0,T]$, where $\sigma_1\in(1,2)\cup(2,3)$ and $\sigma_2\in(\alpha/2,1)\cup(1,2)$ are two regularity parameters.
 Our analysis are   under the weak mesh assumption:
 \begin{itemize}
 \item[{\bf MA.}] There is a constant $C_\gamma>0$ such that $\tau_k\leq C_\gamma\tau\min\{1,t_k^{1-1/\gamma}\}$ for $1\leq k\leq N$, with $t_k\leq C_\gamma t_{k-1}$ and $\tau_k/t_k\leq C_\gamma\tau_{k-1}/t_{k-1}$ for $2\leq k\leq N$,
 \end{itemize}
  where $\gamma\geq1$ is the mesh parameter,   $\tau_k:=t_k-t_{k-1}$ denotes the $k$-th time step size for $1\leq k\leq N$  and $\tau:=\max_{1\leq k\leq N}\{\tau_k\}$.

 We prove that  our method can achieve the desired optimal temporal convergence orders  (see Theorem \ref{Convergence}), that is  ${\cal O}(\tau^{\min\{2-\frac{\alpha}{2},\gamma\sigma_1,\gamma\sigma_2\}})$ for the nonuniform L1 algorithm and  ${\cal O}(\tau^{\min\{2,\gamma\sigma_1,\gamma\sigma_2\}})$ for the nonuniform Alikhanov algorithm.
 We note that the sum-of-exponentials approximations \cite{FL1,Liao-AC2} can also be directly adopted to the proposed nonuniform L1 and Alikhanov algorithms to reduce the memory storage and computational costs. We further design an adaptive time-stepping strategy according to the two kinds of algorithms, which is robust and accurate for dealing with not only the weak initial singularities but also the rapid temporal oscillations of the semilinear diffusion-wave problem.

The main contributions of this paper are summarized below:
\begin{itemize}
\item We propose a novel order reduction method (SFOR) which enables the nonuniform L1 and Alikhanov algorithms for the semilinear diffusion-wave equation can be constructed and analyzed.
\item Based on some reasonable regularity assumptions and weak mesh restrictions, we obtain the optimal convergence orders: the temporal convergence rate is up to $(2-\alpha/2)$-order for the L1 algorithm and second-order for the Alikhanov algorithm.
\item An adaptive time-stepping strategy is designed for the semilinear diffusion-wave equation to efficiently resolve  possible oscillations of the solution.
\end{itemize}

The rest of the paper is organized as follows. In section \ref{SFOR}, we present the novel SFOR method and equivalently rewrite the semilinear diffusion-wave equation into coupled equations. In Section \ref{Algorithms}, we construct and analyze the linearized nonuniform L1 and the nonuniform Alikhanov algorithms, and obtain their optimal convergences unconditionally by $H^2$ energy method. Furthermore, we design an adaptive time-stepping method by combining the proposed nonuniform (fast linearized) L1 and Alikhanov algorithms. Numerical examples are provided in Section \ref{Numerical} to demonstrate the accuracy and efficiency. A brief conclusion is followed in Section \ref{Conclusion}, and the analysis of truncation errors is given in Section \ref{Appendix}.

 Throughout the paper, we use $C$ to denote a generic constant which may depends on the data of the governing problem but is independent of time and space step sizes (or nodes). 

\section{The SFOR method}\label{SFOR}

In this section, we propose a symmetric fractional-order reduction (SFOR) method  such that the technical  framework proposed in \cite{LiaoL1,LiaoGronwall,LiaoSecondOrder,LiaoYanZhang2018} can be adopted to analyze implicit numerical schemes
for solving the diffusion-wave equation  \eqref{eq1} on temporal nonuniform mesh.

The basic idea of the SFOR method is presented in the following lemma.
\begin{lemma}\label{sym-reduction}
For $\alpha\in(1,2)$ and $u(t)\in {\cal C}^1([0,T])\cap{\cal C}^2((0,T])$, it holds that
\begin{align}\label{reduction1}
{\cal D}_t^{\alpha}u(t)={\cal D}_t^{\frac{\alpha}{2}}\left({\cal D}_t^{\frac{\alpha}{2}} u(t)\right)-u'(0)\omega_{2-\alpha}(t).
\end{align}
Moreover, if we take ${\bf u}(t):=u(t)-tu'(0)$, then
\begin{align}\label{reduction2}
{\cal D}_t^\alpha {u}(t)={\cal D}_t^\alpha {\bf u}(t)={\cal D}_t^{\frac{\alpha}{2}}\left({\cal D}_t^{\frac{\alpha}{2}} {\bf u}(t)\right).
\end{align}
\end{lemma}
\begin{proof}
Taking $v(t):={\cal D}_t^{\frac{\alpha}{2}} u(t)$, one has
\begin{align*}
v={\cal D}_t^{\frac{\alpha}{2}} u(t)=({\cal I}^{\frac{\alpha}{2}}u')(t)=&\int_0^t\omega_{1-\frac{\alpha}{2}}(t-s)u'(s)\zd s\\
=&-u'(s)\omega_{2-\frac{\alpha}{2}}(t-s)|_0^t+\int_0^t\omega_{2-\frac{\alpha}{2}}(t-s)u''(s)\zd s\\
=&u'(0)\omega_{2-\frac{\alpha}{2}}(t)+\int_0^t\omega_{2-\frac{\alpha}{2}}(t-s)u''(s)\zd s,
\end{align*}
where the integration by parts has been utilized.
Then
\begin{align*}
v_t=\frac{\zd}{\zd t}{\cal D}_t^{\frac{\alpha}{2}} u(t)
=u'(0)\omega_{1-\frac{\alpha}{2}}(t)+\int_0^t\omega_{1-\frac{\alpha}{2}}(t-s)u''(s)\zd s=u'(0)\omega_{1-\frac{\alpha}{2}}(t)+({\cal I}^{1-\frac{\alpha}{2}}u'')(t).
\end{align*}
Hence, using the composition property ${\cal I}^p{\cal I}^q g(t)={\cal I}^{p+q}g(t)~(p,q>0)$ \cite[pp. 59]{Podnubny}, we get
\begin{align*}
{\cal D}_t^{\frac{\alpha}{2}}\left({\cal D}_t^{\frac{\alpha}{2}} u(t)\right)={\cal D}_t^{\frac{\alpha}{2}}v(t)=({\cal I}^{1-\frac{\alpha}{2}}v')(t)=&
u'(0)({\cal I}^{1-\frac{\alpha}{2}}\omega_{1-\frac{\alpha}{2}})(t)+{\cal I}^{1-\frac{\alpha}{2}}{\cal I}^{1-\frac{\alpha}{2}}u''(t)\\
=&u'(0)\omega_{2-\alpha}(t)+{\cal I}^{2-\alpha}u''(t)\\
=&u'(0)\omega_{2-\alpha}(t)+{\cal D}_t^{\alpha}u(t),
\end{align*}
implying \eqref{reduction1} is true.

The equality \eqref{reduction2} can be obtained directly by taking ${\bf v}:={\cal D}_t^{\frac{\alpha}{2}}{\bf u}(t)$ in the above derivations.
\end{proof}
Now we take
$${\bf u}({\bf x},t):=u({\bf x},t)-t{\tilde \varphi}({\bf x})\quad
\mbox{and} \quad {\bf v}({\bf x},t):={\cal D}_t^{\frac{\alpha}{2}}{\bf u}({\bf x},t).$$
 From \eqref{regularity} and using the Sobolev embedding theorem, we have $\|{\bf u}_t({\bf x},t)\|_\infty\leq C(1+t^{\sigma_1-1})$ for $t\in(0,T]$. Then utilizing the Comparison theorem for integrals (see pp. 400--401 in \cite{Zorich}), one has
  \begin{align*}
  |{\bf v}({\bf x},0)|=&\left|\lim_{t\rightarrow 0}  {\cal D}_t^{\frac{\alpha}{2}} {\bf u}({\bf x},t)\right|\leq
  \frac1{\Gamma(1-\frac{\alpha}{2})}\lim_{t\rightarrow 0} \int_0^t (t-s)^{-\frac{\alpha}{2}}\left|{\bf u}_t({\bf x},s) \right|\zd s\\
  \leq &C\lim_{t\rightarrow 0} \int_0^t (t-s)^{-\frac{\alpha}{2}}(1+s^{\sigma_1-1}) \zd s
  \leq  C\lim_{t\rightarrow 0}(t^{1-\frac{\alpha}{2}}+t^{\sigma_1-\frac{\alpha}{2}})
  =0,
  \end{align*}
which gives ${\bf v}({\bf x},0)=0.$

Thus, by Lemma \ref{sym-reduction}, the equation \eqref{eq1} can be equivalently solved by the following coupled equations:
\begin{align}\label{eq-new1}
& {\cal D}_t^{\frac{\alpha}{2}} {\bf v}=\nu^2\Delta {\bf u}+{f}(u,{\bf x},t)+t\Delta{\tilde \varphi},\quad {\bf x}\in\Omega,~t\in(0,T],\\\label{eq-new2}
& {\bf v}={\cal D}_t^{\frac{\alpha}{2}}{\bf u},\quad {\bf x}\in\Omega,~t\in(0,T],
\end{align}
provided $u={\bf u}+t{\tilde \varphi}$, the initial conditions ${\bf u}({\bf x},0)=\varphi({\bf x}),~ {\bf v}({\bf x},0)=0$ for ${\bf x}\in\Omega$, and boundary conditions ${\bf u}({\bf x},t)={\bf v}({\bf x},t)=0$ for ${\bf x}\in\partial\Omega$.

One can observe that, by utilizing the proposed SFOR method, the explicit orders of the time-fractional derivatives in the resulting coupled equations \eqref{eq-new1} and \eqref{eq-new2} are all  $\alpha/2$. Therefore, they can be discretized by the same strategy (e.g., the L1 or Alikhanov approximations).
\begin{remark}
We observe from our numerical experiments that, by extracting the singular term $u'(0)\omega_{2-\alpha}(t)$ in \eqref{reduction1}, the proposed algorithms will have more regular accuracy due to the regularity of the remaining part. This is the reason why we define the auxiliary function ${\bf v}={\cal D}_t^{\frac{\alpha}{2}}{\bf u}$ with ${\bf u}=u-t{\tilde \varphi}$, instead of ${v}={\cal D}_t^{\frac{\alpha}{2}}{u}$.
\end{remark}

\section{Numerical algorithms}\label{Algorithms}
\subsection{Preliminary}

 Our main concern is the time approximation of \eqref{eq1}.
 Here and hereafter, $g^k$ and $g_h^k$ denotes the numerical approximations of $g(t_k)$ and $g({\bf x}_h,t_k)$, respectively. Define the off-set time points and grid functions
 $$t_{n-\theta}:=\theta t_{n-1}+(1-\theta)t_n\quad \mbox{and}\quad g^{n-\theta}:=\theta g^{n-1}+(1-\theta)g^n,\quad 1\leq n\leq N.$$

 Denote $\beta:=\alpha/2$. The Caputo derivative ${\cal D}_t^{\beta} g(t_{n-\theta})$ can be formally approximated by the following  discrete Caputo derivative with convolution structure:
 \begin{align}\label{dis-Caputo}
 ({\cal D}_\tau^{\beta} g)^{n-\theta}:=\sum_{k=1}^nA_{n-k}^{(n)}\nabla_{\tau} g^k, \quad \mbox{where}~\nabla_{\tau} g^k=g^k-g^{k-1}.
 \end{align}
The general discretization \eqref{dis-Caputo} includes two practical ones. It leads to the L1 formula while $\theta=0$ (see also \eqref{L1-discretization})  and yields the Alikhanov formula while $\theta=\beta/2$ (see also \eqref{AAA-discretization}).
 To efficiently solve the semilinear diffusion-wave equation with possible weak singular or more complicated solutions,
 we next give more explicit formulations of these two classical approximations on possible nonuniform meshes, which have also been rigorously studied in \cite{LiaoL1,LiaoGronwall,LiaoSecondOrder}.

{\bf Nonuniform L1 formula.} The L1 formula on general mesh for the approximation of the Caputo derivative ${\cal D}_t^{\beta}g(t_n)$ is given as:
\begin{align}\label{L1-discretization}
({\cal D}_\tau^{\beta}g)^n:=\sum_{k=1}^n\int_{t_{k-1}}^{t_k}\omega_{1-\beta}(t_n-s)(\Pi_{1,k}g(s))'\zd s=\sum_{k=1}^n A_{n-k}^{(n)}\nabla_\tau g^k,
\end{align}
where $\Pi_{1,k}$ represents the linear interpolation operator, and
\begin{align}\label{L1coe}
A_{n-k}^{(n)}:=\int_{t_{k-1}}^{t_k}\frac{\omega_{1-\beta}(t_n-s)}{\tau_k}\zd s.
\end{align}

 {\bf Nonuniform Alikhanov formula.} Denote $\theta:=\beta/2=\alpha/4$, and define the discrete coefficients
\begin{align*}
&a_{n-k}^{(n)}:=\frac1{\tau_k}\int_{t_{k-1}}^{\min\{t_k,t_{n-\theta}\}}\omega_{1-\beta}(t_{n-\theta}-s)\zd s ,~1\leq k\leq n;\\
&b_{n-k}^{(n)}:=\frac2{\tau_k(\tau_k+\tau_{k+1})}\int_{t_{k-1}}^{t_k}\omega_{1-\beta}(t_{n-\theta}-s)(s-t_{k-\frac12})\zd s ,~1\leq k\leq n-1.
\end{align*}
Referring to \cite{LiaoSecondOrder}, the Alikhanov formula on general mesh for the approximation of the Caputo derivative ${\cal D}_t^{\beta}g(t_{n-\theta})$ is
\begin{align}\nonumber
({\cal D}_\tau^{\beta}g)^{n-\theta}:=&\sum_{k=1}^{n-1}\int_{t_{k-1}}^{t_k}\omega_{1-\beta}(t_{n-\theta}-s)(\Pi_{2,k}g(s))'\zd s+\int_{t_{n-1}}^{t_{n-\theta}}\omega_{1-\beta}(t_{n-\theta}-s)(\Pi_{1,n}g(s))'\zd s\\\label{AAA-discretization}
=&\sum_{k=1}^n A_{n-k}^{(n)}\nabla_\tau g^k,
\end{align}
where $\Pi_{2,k}$ denotes the quadratic interpolation operator, and the discrete convolution kernels $A_{n-k}^{(n)}$ here are given as follows: $A_0^{(1)}:=a_0^{(1)}$ for $n=1$, and
\begin{align}\label{L2coe}
A_{n-k}^{(n)}:=\left\{\begin{array}{ll}
a_0^{(n)}+\rho_{n-1}b_1^{(n)},\quad & k=n,\\
a_{n-k}^{(n)}+\rho_{k-1}b_{n-k+1}^{(n)}-b_{n-k}^{(n)}, & 2\leq k\leq n-1,\\
a_{n-1}^{(n)}-b_{n-1}^{(n)},  & k=1,
\end{array}\right.\quad \mbox{for}~n\geq 2,
\end{align}
with $\rho_k:=\tau_k/\tau_{k+1}$  and $\rho:=\max_{k}\{\rho_k\}$ being the local time step-size ratios and the maximum ratio, respectively.
\begin{remark}
In the rest of this paper, we will use the general form \eqref{dis-Caputo} to represent the nonuniform L1 formula  and Alikhanov formula.
The discrete coefficients $A_{n-k}^{(n)}$ and the related properties studied later  correspondingly refer to those of the nonuniform L1 formula and the Alikhanov formula while $\theta=0$ and $\theta=\beta/2$, respectively.
\end{remark}

The following two basic properties have been verified in \cite{LiaoGronwall,LiaoSecondOrder} for the discrete coefficients of the nonuniform L1 formula (with $\pi_A=1$) and the nonuniform Alikhanov formula (with $\pi_A=11/4$ and $\rho=7/4$), which are  required in the numerical analysis of corresponding algorithms:
\begin{itemize}
\item[{\bf A1.}] The discrete kernels are positive and monotone: $A_0^{(n)}\geq A_1^{(n)}\geq \cdots\geq A_{n-1}^{(n)}>0$;
\item[{\bf A2.}] There is a constant $\pi_A>0$ such that $A_{n-k}^{(n)}\geq \frac1{\pi_A}\int_{t_{k-1}}^{t_k}\frac{\omega_{1-\beta}(t_n-s)}{\tau_k}\zd s$ for $1\leq k\leq n\leq N$.
\end{itemize}

With {\bf A1}--{\bf A2}, a natural and important property is valid  for the nonuniform L1 formula  \cite[proof of Theorem 2.1]{LiaoL1} and the nonuniform Alikhanov formula \cite[Corollary 2.3]{LiaoSecondOrder}:
\begin{align}\label{product-proper1}
\left\langle ({\cal D}_\tau^\beta g)^{n-\theta}, g^{n-\theta} \right\rangle \geq \frac12 \sum_{k=1}^nA_{n-k}^{(n)}\nabla_{\tau}(\|g^k\|^2) \quad \mbox{for}~1\leq n\leq N.
\end{align}

A discrete fractional Gr\"{o}nwall inequality proposed in \cite[Theorem 3.1]{LiaoGronwall} is a crucial tool in the numerical analysis of fractional problems. As required in the analysis later, we present a slightly modified version in the following. It is easy to trace the proof of  \cite[Theorem 3.1]{LiaoGronwall} to justify the modification, here we skip its trivial derivations.
\begin{lemma}\label{gronwall-lemma}
Let $(g^n)_{n=1}^N$ and $(\lambda_l)_{l=0}^{N-1}$ be given nonnegative sequences. Assume that there exists a constant $\Lambda$ (independent of the step sizes) such that $\Lambda\geq\sum_{l=0}^{N-1}\lambda_l$, and that the maximum step size satisfies
$$\max_{1\leq n\leq N}\tau_n\leq \frac1{{^\beta\sqrt{4\pi_A\Gamma(2-\beta)\Lambda}}}.$$
Then, for any nonnegative sequence $(v^k)_{k=0}^N$ and $(w^k)_{k=0}^N$  satisfying
\begin{align*}
\sum_{k=1}^nA_{n-k}^{(n)}\nabla_{\tau} \left[(v^k)^2+(w^k)^2 \right] \leq \sum_{k=1}^n \lambda_{n-k}\left(v^{k-\theta}+w^{k-\theta}\right)^2+ (v^{n-\theta}+w^{n-\theta})g^{n}, ~ 1\leq n\leq N,
\end{align*}
it holds that
\begin{align}\label{gronwall2}
v^n+w^n\leq  4E_{\beta}(4\max(1,\rho)\pi_A\Lambda t_n^\beta)\left(v^0+w^0+\max_{1\leq k\leq n}\sum_{j=1}^kP_{k-j}^{(k)}g^{j} \right)\quad \mbox{for}~1\leq n\leq N,
\end{align}
where $E_{\beta}(z)=\sum_{k=0}^\infty\frac{z^k}{\Gamma(1+k\beta)}$ is the Mittag-Leffler function.
\end{lemma}

The coefficients $P_{n-j}^{(n)}$ in \eqref{gronwall2} are called the complementary discrete kernels (\cite{LiaoGronwall}) which are defined on the convolution coefficients $A_{n-k}^{(n)}$ :
\begin{align}\label{Pnk}
P_0^{(n)}:=\frac1{A_0^{(n)}},\quad P_{n-j}^{(n)}:=\frac1{A_0^{(j)}}\sum_{k=j+1}^n\left(A_{k-j-1}^{(k)}-A_{k-j}^{(k)} \right)P_{n-k}^{(n)},\quad 1\leq j\leq n-1.
\end{align}
It has been shown in \cite[Lemmma 2.1]{LiaoGronwall} that the kernels satisfy
\begin{align}\label{P-proper}
0\leq P_{n-j}^{(n)}\leq \pi_A\Gamma(2-\beta)\tau_j^\beta,\quad \sum_{j=1}^nP_{n-j}^{(n)}\omega_{1-\beta}(t_j)\leq \pi_A,\quad 1\leq j\leq n\leq N.
\end{align}

\subsection{Nonuniform L1 and Alikhanov algorithms}\label{L1scheme}

We now implement linearized algorithms on temporal nonuniform meshes to solve the coupled equations \eqref{eq-new1}--\eqref{eq-new2} based on the nonuniform L1 and Alikhanov formulas.

 Some basic notations in the spatial direction are needed. The uniform spatial step sizes are denoted by $h_x:=(x_r-x_l)/M_x$ and $h_y:=(y_r-y_l)/M_y$ respectively, where $M_x,M_y$ are positive integers. The mesh space is given by $\bar\Omega_h:=\{{\bf x}_h=(x_l+ih_x,y_l+jh_y)|0\leq i\leq M_x,0\leq j\leq M_y\}$. For any grid functions $u_h:=\{u_{i,j}=u(x_i,t_j)|(x_i,t_j)\in\bar\Omega_h \}$, we employ standard five-point finite difference operator $\Delta_h:=\delta_x^2+\delta_y^2$ on $\bar\Omega_h$ to discretize the Laplacian operator $\Delta$, where
         $\delta_x^2 u_{i,j}:=(u_{i+1,j}-2u_{i,j}+u_{i-1,j})/{h_x^2}$
         and $\delta_y^2 u_{i,j}$ is defined similarly.

Denote
$$F(u_h^{n-\theta}):=f(u_h^{n-1},{\bf x}_h,t_{n-\theta})+
 (1-\theta)\partial_uf(u_h^{n-1},{\bf x}_h,t_{n-\theta})(u_h^n-u_h^{n-1}),\quad 1\leq n\leq N.$$
The linearized and implicit difference schemes which based on the L1 and the Alikhanov approximations on general nonuniform temporal meshes for the problem \eqref{eq-new1}--\eqref{eq-new2} or the problem \eqref{eq1} are constructed as follows:
\begin{align}\label{sc1}
&({\cal D}_\tau^{\beta} {\bf v}_h)^{n-\theta}=\nu^2\Delta_h {\bf u}_h^{n-\theta}+F({u}_h^{n-\theta})+ t_{n-\theta}\Delta{\tilde \varphi}_h, \quad {\bf x}\in \Omega_h, ~1\leq n\leq N;\\\label{sc2}
&  {\bf v}_h^{n-\theta}=({\cal D}_\tau^{\beta} {\bf u}_h)^{n-\theta}, \quad {\bf x}\in \Omega_h, ~1\leq n\leq N;\\\label{sc3}
& u_h^n={\bf u}_h^n+t_n{\tilde \varphi}_h, \quad {\bf x}\in \Omega_h, ~0\leq n\leq N,
\end{align}
equipped with the initial conditions ${\bf u}_h^0=\varphi({\bf x}_h)$ and $ {\bf v}_h^0=0$ for $  {\bf x}\in \Omega_h$, and the boundary conditions ${\bf u}_h^n= {\bf v}_h^n=0$ for ${\bf x}\in \partial\Omega_h,~ 1\leq n\leq N$.
\begin{remark}
The equations \eqref{sc1}--\eqref{sc3} represent two different numerical algorithms for solving the semilinear diffusion-wave equation. It is the nonuniform L1 algorithm while $\theta=0$ and is the nonuniform Alikhanov algorithm while $\theta=\beta/2=\alpha/4$.
\end{remark}
In order to analyze the two proposed  algorithms, we consider an equivalent form of \eqref{sc1}--\eqref{sc3}. Firstly, denote $w:={\cal D}_t^{\beta}{\bf v}-f(u,{\bf x},t)+t\Delta{\tilde \varphi}$ with the initial condition $w({\bf x},0):=\nu^2\Delta \varphi$ and the boundary $w({\bf x},t):=-f(0,{\bf x},t)$. Then \eqref{eq-new1}--\eqref{eq-new2} can be rewritten as
\begin{align*}
&w={\cal D}_t^{\beta}{\bf v}-f(u,{\bf x},t)+t\Delta{\tilde \varphi},\quad {\bf x}\in\Omega,~t\in(0, T];\\
&w=\nu^2\Delta {\bf u}, \quad {\bf x}\in \Omega,~t\in(0, T];\\
&{\bf v}={\cal D}_t^{\beta}{\bf u}, \quad {\bf x}\in\Omega,~t\in(0, T].
\end{align*}
Utilizing the nonuniform L1 and Alikhanov formulas and the linearized technique  to approximate the above equations, we obtain an auxiliary system of \eqref{sc1}--\eqref{sc3}:
\begin{align}\label{sc1-w}
&w_h^{n-\theta}=({\cal D}_\tau^{\beta}{\bf v}_h)^{n-\theta}-F(u_h^{n-\theta})+t_{n-\theta}\Delta{\tilde \varphi}_h,\quad {\bf x}_h\in\Omega_h,~1\leq n\leq N;\\\label{sc2-w}
&w_h^n=\nu^2\Delta_h {\bf u}_h^n, \quad {\bf x}_h\in \Omega_h,~0\leq n\leq N;\\\label{sc3-w}
&{\bf v}_h^{n-\theta}=({\cal D}_\tau^{\beta}{\bf u}_h)^{n-\theta}, \quad {\bf x}_h\in\Omega_h,~1\leq n\leq N;\\\label{sc4-w}
& u_h^n={\bf u}_h^n+t_n{\tilde \varphi}_h, \quad {\bf x}\in \Omega_h, ~0\leq n\leq N.
\end{align}
As $w_h^{n-\theta}=(1-\theta)w_h^n+\theta w_h^{n-1}$, one can see that equations \eqref{sc1}--\eqref{sc3} are equivalent to \eqref{sc1-w}--\eqref{sc4-w} by eliminating the  functions $w_h^{n-\theta}$ and $w_h^n$.

\subsection{Unconditional convergence}
In this subsection, we show the unconditional convergence of the proposed nonuniform L1 and Alikhanov algorithms \eqref{sc1}--\eqref{sc3} according to their auxiliary system \eqref{sc1-w}--\eqref{sc4-w}.

 Take $\Omega_h=\bar\Omega_h\cap \Omega$ and $\partial\Omega_h=\bar\Omega_h\cap \partial\Omega$. For $u_h,v_h$ belonging to the space of grid functions which vanish on $\partial\Omega_h$, we introduce the discrete inner product $\langle u,v \rangle:=h_xh_y\sum_{{\bf x}_h\in\Omega_h}u_hv_h$, the discrete $L_2$-norm
  $ \|u\|:=\sqrt{\langle u,u \rangle}$, the discrete $L_\infty$-norm $\|u\|_{\infty}:=\max\{|u_h|\}$,
  the discrete $H^1$ seminorms $\|\delta_x u\|$ and $\|\delta_y u\|$,
 and $\|\nabla_h u\|:=\sqrt{\|\delta_x u\|^2+\|\delta_y u\|^2}$, where $\delta_xu_{i-\frac12,j}:=(u_{i,j}-u_{i-1,j})/h_x$ and similar definition works for $\delta_y u_{i,j-\frac12}$. One can easily  check that
 $\langle \Delta_h u,u \rangle=-\|\nabla_h u\|^2$, and, for some positive constants ${\tilde C}_\Omega,{\hat C}_\Omega$,
  the embedding  inequalities are valid (\cite{Liao-NMPDE2010}): $\|u\| \leq {\tilde C}_\Omega\|\nabla_h u\|$ and  $\max\{\|\nabla_h u\|,\|u\|_\infty\}\leq {\hat C}_\Omega\|\Delta_h u\|$. For simplicity of presentation, we take $h:=\max\{ h_x,h_y\}$ and $C_\Omega=\max\{{\tilde C}_\Omega,{\hat C}_\Omega\}$.

For $\vartheta\in(0,1]$, let  $U_h^{n-\vartheta}:=u({\bf x}_h,t_{n-\vartheta})$  and denote the solution errors
$${\tilde u}_h^n:=U_h^n-u_h^n={\bf u}({\bf x}_h,t_n)-{\bf u}_h^n,\quad {\tilde v}_h^n:={\bf v}({\bf x}_h,t_n)-{\bf v}_h^n,\quad \mbox{and}\quad  {\tilde w}_h^n:=w({\bf x}_h,t_n)-w_h^n.$$
One can obtain the error system of \eqref{sc1-w}--\eqref{sc4-w}:
\begin{align}\label{error_eq1}
&{\tilde w}_h^{n-\theta}=({\cal D}_\tau^{\beta}{\tilde v}_h)^{n-\theta}-{\cal N}_h^{n-\theta}-({\cal T}_f)_h^{n-\theta}+({\cal T}_{v1})_h^{n-\theta}-({\cal T}_w)_h^{n-\theta},\quad {\bf x}_h\in\Omega_h,~1\leq n\leq N;\\\label{error_eq2}
&{\tilde w}_h^n=\nu^2\Delta_h {\tilde u}_h^n+\nu^2{\cal S}_h^{n}, \quad {\bf x}_h\in \Omega_h,~1\leq n\leq N;\\\label{error_eq3}
&{\tilde v}_h^{n-\theta}=({\cal D}_\tau^{\beta}{\tilde u}_h)^{n-\theta}+({\cal T}_{u})_h^{n-\theta}-({\cal T}_{v2})_h^{n-\theta}, \quad {\bf x}_h\in\Omega_h,~1\leq n\leq N;\\\nonumber
&{\tilde u}_h^0={\tilde v}_h^0={\tilde w}_h^0=0,\quad {\bf x}_h\in\bar\Omega_h; \qquad {\tilde u}_h^n={\tilde v}_h^n=0,\quad {\bf x}_h\in\partial\Omega_h,~1\leq n\leq N.
\end{align}
where $({\cal T}_f)_h^{n-\theta}$, $({\cal T}_{v1})_h^{n-\theta}$, $({\cal T}_w)_h^{n-\theta}$, $({\cal T}_{u})_h^{n-\theta}$, $({\cal T}_{v2})_h^{n-\theta}$ and ${\cal S}_h^{n}$ are the temporal and spatial truncation errors, see more details in  Appendix (Section \ref{Appendix});
and
\begin{align*}
{\cal N}_h^{n-\theta}:=&F(U_h^{n-\theta})-F(u_h^{n-\theta})\\
=&(1-\theta) \left[ \partial_uf(u_h^{n-1},{\bf x}_h,t_{n-\theta})\nabla_\tau{\tilde u}_h^n+{\tilde u}_h^{n-1}\nabla_\tau U_h^n\int_0^1\partial^2_uf(sU_h^{n-1}+(1-s)u_h^{n-1},{\bf x}_h,t_{n-\theta})\zd s \right]\\
&+{\tilde u}_h^{n-1}\int_0^1\partial_uf(sU_h^{n-1}+(1-s)u_h^{n-1},{\bf x}_h,t_{n-\theta})\zd s.
\end{align*}
It can be deduced from \eqref{error_eq2} that
$\sum_{k=1}^nA_{n-k}^{(n)}\nabla_\tau {\tilde w}_h^k= \nu^2\sum_{k=1}^nA_{n-k}^{(n)}\nabla_\tau\left( \Delta_h {\tilde u}_h^k+{\cal S}_h^{k} \right)$. 
Then, performing  the operator $\Delta_h$ on \eqref{error_eq1} and \eqref{error_eq3}, it follows
\begin{align}\nonumber
&\Delta_h{\tilde w}_h^{n-\theta}=({\cal D}_\tau^{\beta}\Delta_h{\tilde v}_h)^{n-\theta}+\Delta_h{\cal N}_h^{n-\theta}-\Delta_h({\cal T}_f)_h^{n-\theta}+\Delta_h({\cal T}_{v1})_h^{n-\theta}-\Delta_h({\cal T}_w)_h^{n-\theta},\\\nonumber
&~~~~~~~~~~~~~~~~~~~~~~~~~~~~~~~~~~~~~~~~~~~~~~~~~~~~~~~~~~~~~~~~~~~~~~~~~ {\bf x}_h\in\Omega_h,~1\leq n\leq N;\\\label{error_eq5}
&({\cal D}_\tau^{\beta}{\tilde w}_h)^{n-\theta}=\nu^2({\cal D}_\tau^{\beta}\Delta_h {\tilde u}_h)^{n-\theta}+\nu^2({\cal D}_\tau^{\beta}{\cal S}_h)^{n-\theta}, \quad {\bf x}_h\in \Omega_h,~1\leq n\leq N;\\\label{error_eq6}
&\Delta_h{\tilde v}_h^{n-\theta}=({\cal D}_\tau^{\beta}\Delta_h{\tilde u}_h)^{n-\theta}+\Delta_h({\cal T}_{u})_h^{n-\theta}-\Delta_h({\cal T}_{v2})_h^{n-\theta}, \quad {\bf x}_h\in\Omega_h,~1\leq n\leq N;\\\nonumber
&{\tilde u}_h^0={\tilde v}_h^0={\tilde w}_h^0=0,\quad {\bf x}_h\in\bar\Omega_h; \qquad {\tilde u}_h^n={\tilde v}_h^n=0,\quad {\bf x}_h\in\partial\Omega_h,~1\leq n\leq N.
\end{align}
By eliminating the term $({\cal D}_\tau^{\beta}\Delta_h{\tilde u}_h)^{n-\theta}$ in \eqref{error_eq5} and \eqref{error_eq6}, we get
\begin{align}\nonumber
&\Delta_h{\tilde w}_h^{n-\theta}=({\cal D}_\tau^{\beta}\Delta_h{\tilde v}_h)^{n-\theta}+\Delta_h{\cal N}_h^{n-\theta}-\Delta_h({\cal T}_f)_h^{n-\theta}+\Delta_h({\cal T}_{v1})_h^{n-\theta}-\Delta_h({\cal T}_w)_h^{n-\theta},\\\label{error_eq7}
&~~~~~~~~~~~~~~~~~~~~~~~~~~~~~~~~~~~~~~~~~~~~~~~~~~~~~~~~~~~~~~~~~~~~~~~~~  {\bf x}_h\in\Omega_h,~1\leq n\leq N;\\\label{error_eq8}
&\frac1{\nu^2}({\cal D}_\tau^{\beta}{\tilde w}_h)^{n-\theta}=\Delta_h{\tilde v}_h^{n-\theta}-\Delta_h({\cal T}_{u})_h^{n-\theta}+\Delta_h({\cal T}_{v2})_h^{n-\theta}+({\cal D}_\tau^{\beta}{\cal S}_h)^{n-\theta}, \quad {\bf x}_h\in \Omega_h,~1\leq n\leq N;\\\nonumber
&{\tilde u}_h^0={\tilde v}_h^0={\tilde w}_h^0=0,\quad {\bf x}_h\in\bar\Omega_h; \qquad {\tilde u}_h^n={\tilde v}_h^n=0,\quad {\bf x}_h\in\partial\Omega_h,~1\leq n\leq N.
\end{align}

\begin{lemma}\label{Delta-F}
Let ${\cal F}(\psi({\bf x}),{\bf x})\in C^2(\mathbb{R}\times\Omega)$, and $\{\psi_h\}$ be a grid function  which satisfy \\$\max\{\|\psi\|_\infty,\|\nabla_h\psi\|,\|\Delta_h\psi\|\}\leq C_\psi$. Then there is a constant $C_F>0$ dependent on $C_\psi$ and $C_\Omega$ such that
$$\|\Delta_h[{\cal F}(\psi,{\bf x})v]\|\leq C_F\|\Delta_h v\|.$$
\end{lemma}
\begin{proof}
The proof can be worked out following that of  \cite[Lemma 4.1]{LiaoYanZhang2018} just by routine computations on $\delta_x {\cal F}(\psi_{i-\frac12,j},(x_{i-\frac12},y_j))$, $\delta_y {\cal F}(\psi_{i,j-\frac12},(x_{i},y_{j-\frac12}))$, $\delta_x^2{\cal F}(\psi_{i,j},(x_{i},y_j))$ and $\delta_y^2{\cal F}(\psi_{i,j},(x_{i},y_j))$  using the Taylor formula with integral remainder.
\end{proof}

We next show the unconditional convergence of the proposed linearized scheme \eqref{sc1}--\eqref{sc3} based on the $H^2$ energy method (\cite{LiaoYanZhang2018,Liao-NMPDE2010}).
\begin{theorem}\label{Convergence}
Let $f\in C^{(4,2,0)}({\mathbb R}\times\Omega\times[0,T])$. If the assumptions in \eqref{regularity} and the mesh assumption {\bf MA} hold, the linearized schemes \eqref{sc1}--\eqref{sc3} are unconditional convergent with
\begin{equation}\label{convergence-result}
\|\Delta_h {\tilde u}^n\|+\|\nabla_h {\tilde v}^n\|\leq \left\{
\begin{array}{ll}\smallskip
C(\tau^{\min\{2-\beta,\gamma\sigma_1,\gamma\sigma_2\}}+h^2),\quad \mbox{if}\quad \theta=0;\\
C(\tau^{\min\{2,\gamma\sigma_1,\gamma\sigma_2\}}+h^2),\quad \mbox{if}\quad \theta=\frac{\beta}{2};
\end{array}
\quad \mbox{for}\quad1\leq n\leq N.
\right.
\end{equation}
\end{theorem}

\begin{proof}
Taking inner product of equations \eqref{error_eq7} and \eqref{error_eq8} with ${\tilde v}_h^{n-\theta}$ and ${\tilde w}_h^{n-\theta}$ respectively, we have
\begin{align}\nonumber
\left\langle \Delta_h{\tilde w}^{n-\theta},{\tilde v}^{n-\theta} \right\rangle
=&\left\langle ({\cal D}_\tau^{\beta}\Delta_h{\tilde v})^{n-\theta},{\tilde v}^{n-\theta}\right\rangle\\\label{con-inner1}
&+\left\langle\Delta_h{\cal N}^{n-\theta} -\Delta_h({\cal T}_f)^{n-\theta}+\Delta_h({\cal T}_{v1})^{n-\theta}-\Delta_h({\cal T}_w)^{n-\theta},{\tilde v}^{n-\theta}\right\rangle
\end{align}
and
\begin{align}\label{con-inner2}
&\left\langle \frac1{\nu^2}({\cal D}_\tau^{\beta}{\tilde w})^{n-\theta},{\tilde w}^{n-\theta}\right\rangle=\left\langle\Delta_h{\tilde v}^{n-\theta},{\tilde w}^{n-\theta}\right\rangle+\left\langle-\Delta_h({\cal T}_{u})^{n-\theta}+\Delta_h({\cal T}_{v2})^{n-\theta}+({\cal D}_\tau^{\beta}{\cal S})^{n-\theta},{\tilde w}^{n-\theta}\right\rangle.
\end{align}
With the identity $\left\langle \Delta_h{\tilde w}^{n-\theta},{\tilde v}^{n-\theta} \right\rangle=\left\langle\Delta_h{\tilde v}^{n-\theta},{\tilde w}^{n-\theta}\right\rangle$ and the zero boundary conditions of ${\tilde v}_h^{n-\theta}$, it follows form \eqref{con-inner1}--\eqref{con-inner2} that
\begin{align*}
&\left\langle  \frac1{\nu^2}({\cal D}_\tau^{\beta}{\tilde w})^{n-\theta},{\tilde w}^{n-\theta}\right\rangle+\left\langle ({\cal D}_\tau^{\beta}\nabla_h{\tilde v})^{n-\theta},\nabla_h{\tilde v}^{n-\theta}\right\rangle\\
=&\left\langle\Delta_h{\cal N}^{n-\theta}  -\Delta_h({\cal T}_f)^{n-\theta}+\Delta_h({\cal T}_{v1})^{n-\theta}-\Delta_h({\cal T}_w)^{n-\theta},{\tilde v}^{n-\theta}\right\rangle\\
&+\left\langle-\Delta_h({\cal T}_{u})^{n-\theta}+\Delta_h({\cal T}_{v2})^{n-\theta}+({\cal D}_\tau^{\beta}{\cal S})^{n-\theta},{\tilde w}^{n-\theta}\right\rangle.
\end{align*}
Utilizing \eqref{product-proper1} and the Cauchy-Schwarz inequality, the above equation leads to
\begin{align}\nonumber
&\sum_{k=1}^nA_{n-k}^{(n)}\nabla_\tau(\|{\tilde w}^n\|^2+\nu^2\|\nabla_h{\tilde v}^n\|^2)\\\nonumber
\leq&2\nu^2\left(\|\Delta_h{\cal N}^{n-\theta}\|+\| \Delta_h({\cal T}_f)^{n-\theta}\|+\|\Delta_h({\cal T}_{v1})^{n-\theta}\|+\|\Delta_h({\cal T}_w)^{n-\theta}\|\right)\| {\tilde v}^{n-\theta}\|\\\nonumber
&+2\nu^2\left(\|\Delta_h({\cal T}_{u})^{n-\theta}\|+\|\Delta_h({\cal T}_{v2})^{n-\theta}\|+\|({\cal D}_\tau^{\beta}{\cal S})^{n-\theta}\|\right)\|{\tilde w}^{n-\theta}\|\\\label{con-est1}
\leq& 2\nu(1+\nu)(1+C_\Omega)\left(\|\Delta_h{\cal N}^{n-\theta}\|+ {\cal T}^{n-\theta}   \right)\left(\|{\tilde w}^{n-\theta}\| +\nu\| \nabla_h{\tilde v}^{n-\theta}\| \right),
\end{align}
where
$${\cal T}^{n-\theta}:=\| \Delta_h({\cal T}_f)^{n-\theta}\|+\|\Delta_h({\cal T}_{v1})^{n-\theta}\|+\|\Delta_h({\cal T}_w)^{n-\theta}\| +\|\Delta_h({\cal T}_{u})^{n-\theta}\|+\|\Delta_h({\cal T}_{v2})^{n-\theta}\|+\|({\cal D}_\tau^{\beta}{\cal S})^{n-\theta}\|.$$
From \eqref{PTf} and \eqref{PDSn}--\eqref{PTv2}, there exist positive constant $C_r$ such that
\begin{equation}\label{PTtheta}
\sum_{j=1}^n P_{n-j}^{(n)}{\cal T}^{j-\theta} \leq \left\{
\begin{array}{ll}\smallskip
C_r (\tau^{\min\{2-\beta,\gamma\sigma_1,\gamma\sigma_2\}}+h^2), \quad \mbox{for}\quad \theta=0;\\
C_r (\tau^{\min\{2,\gamma\sigma_1,\gamma\sigma_2\}}+h^2), \quad \mbox{for}\quad \theta=\frac{\beta}{2}.
\end{array}\right.
\end{equation}
 From the regularity assumptions in \eqref{regularity}, we introduce the following constant
\begin{align}\label{Un-assumption}
C_0=\max_{0\leq n\leq N}\{\|U^n \|_\infty, \|\nabla_h U^n \|, \|\Delta_h U^n \| \}.
\end{align}
The mathematical induction method will be applied to show that
\begin{align}\label{induction-result}
\|{\tilde w}^n\|+\nu\|\nabla_h{\tilde v}^n\| \leq {\cal E}_n {\tilde {\cal T}}^{n-\theta},\quad 1\leq n\leq N,
\end{align}
where ${\cal E}_n :=4E_\beta(2\max(1,\rho)\pi_A \Lambda t_n^\beta)$ with $\Lambda=2\nu(1+\nu)C_f(1+C_\Omega)$ and
\begin{equation}\nonumber
{\tilde {\cal T}}^{n-\theta}:=2\nu(1+\nu)(1+C_\Omega)\times\left\{
\begin{array}{ll}\smallskip
C_r(\tau^{\min\{2-\beta,\gamma\sigma_1,\gamma\sigma_2\}}+h^2) +\pi_A\Gamma(1-\beta)C_fC_ut_n^\beta h^2, \quad \mbox{for}\quad \theta=0;\\
C_r(\tau^{\min\{2,\gamma\sigma_1,\gamma\sigma_2\}}+h^2) +\pi_A\Gamma(1-\beta)C_fC_ut_n^\beta h^2, \quad \mbox{for}\quad \theta=\frac{\beta}{2},
\end{array}\right.
\end{equation}
in which $C_f=\max\{(1-\theta)C_1,C_2[1+((1-\theta)(C_2+1)+1)/\theta]\}$ with $C_1$ and $C_2$ being two proper positive constants which depend on $C_0$ and $C_\Omega$.

While $n=1$, it holds that ${\tilde u}_h^0=0$ and $u_h^0=U_h^0\leq C_0$. Suppose $f\in C^{(3,2,0)}({\mathbb R}\times\Omega\times[0,T])$, by  Lemma \ref{Delta-F} and \eqref{Sn}, there exists a positive constant $C_1$ such that
\begin{align}\label{N1theta}
\|\Delta_h{\cal N}^{1-\theta}\|=(1-\theta) \|\Delta_h f_u'(u_h^0,{\bf x},t_{1-\theta}){\tilde u}_h^1\|\leq (1-\theta)C_1\|\Delta_h{\tilde u}_h^1\|\leq (1-\theta)C_1(\|{\tilde w}^1\|+C_uh^2).
\end{align}
For simplicity, denote $\|{\tilde w}^{(n-\theta)}\|:=(1-\theta)\|{\tilde w}^n\|+\theta\|{\tilde w}^{n-1}\|$.
Similarly, we define $\|\nabla_h{\tilde v}^{(n-\theta)}\|$. The triangle inequality gives $\|{\tilde w}^{n-\theta}\|\leq \|{\tilde w}^{(n-\theta)}\|$ and $\|\nabla_h{\tilde v}^{n-\theta}\|\leq \|\nabla_h{\tilde v}^{(n-\theta)}\|$.

Then, it follows from \eqref{con-est1} and \eqref{N1theta} that
\begin{align*}
A_{0}^{(1)}\nabla_\tau(\|{\tilde w}^1\|^2+\nu^2\|\nabla_h{\tilde v}^1\|^2)
\leq 2\nu(1+\nu)C_1(1+C_\Omega) \left( \|{\tilde w}^{(1-\theta)}\|+\nu\|\nabla_h{\tilde v}^{(1-\theta)}\|\right)^2\\
+2\nu(1+\nu)(1+C_\Omega)\left({\cal T}^{1-\theta}+(1-\theta)C_1C_uh^2\right) \left( \|{\tilde w}^{(1-\theta)}\|+\nu\|\nabla_h{\tilde v}^{(1-\theta)}\|\right).
\end{align*}
Thus, applying Lemma \ref{gronwall-lemma} on the above inequality, and utilizing \eqref{PTtheta}, we get
\begin{align}\nonumber
\|{\tilde w}^1\|+\nu\|\nabla_h{\tilde v}^1\|
\leq {\cal E}_1
\left[ 2\nu(1+\nu)(1+C_\Omega)P_{0}^{(1)} \left({\cal T}^{1-\theta}+(1-\theta)C_1C_uh^2\right) \right]\leq {\cal E}_1 {\tilde {\cal T}}^{1-\theta},
\end{align}
which means that \eqref{induction-result} holds for $n=1$.

Assume that \eqref{induction-result} is valid for $1\leq k\leq n-1~(n\geq 2)$. The eq. \eqref{error_eq2} and discrete embedding inequalities imply that
\begin{align*}
\max\{ \| {\tilde u}^k\|_\infty,\|\nabla_h {\tilde u}^k\|, \|\Delta_h {\tilde u}^k\| \}\leq &\max\{1,C_\Omega\} \left(\frac1{\nu^2}\|{\tilde w}^k\| +\|{\cal S}^k\|\right)\\
\leq &\max\{1,C_\Omega\} \left(\frac1{\nu^2}{\cal E}_k {\tilde {\cal T}}^{k-\theta}+C_u h^2\right)\leq 1,
\end{align*}
for $ 1\leq k\leq n-1$ and small step sizes. So according to \eqref{Un-assumption}, the numerical solutions satisfy
$$\max\{ \| {u}^k\|_\infty,\|\nabla_h {u}^k\|, \|\Delta_h {u}^k\| \}\leq C_0+1.$$
Now for $k=n$, suppose $f\in C^{(4,2,0)}({\mathbb R}\times\Omega\times[0,T])$.
By Lemma \ref{Delta-F} there exists a positive constant $C_2$ such that
\begin{align}\nonumber
\| \Delta_h {\cal N}^{n-\theta} \| \leq& (1-\theta)\bigg(\| \Delta_h [f_u'(u^{n-1},{\bf x},t_{n-\theta})\nabla_\tau {\tilde u}^n]\| \\\nonumber
&+\int_0^1 \| \Delta_h[f_u''(sU^{n-1}+(1-s)u^{n-1},{\bf x},t_{n-\theta}) {\tilde u}^{n-1}\nabla_\tau U^n]\|\zd s \bigg)\\\nonumber
&+\int_0^1 \| \Delta_h[f_u'(sU^{n-1}+(1-s)u^{n-1},{\bf x},t_{n-\theta}) {\tilde u}^{n-1}]\|\zd s \\\nonumber
\leq& (1-\theta)C_2\left(\| \Delta_h (\nabla_\tau{\tilde u}^n)\|+C_2\| \Delta_h {\tilde u}^{n-1}\|\right)+C_2\|\Delta_h {\tilde u}^{n-1}\| \\\nonumber
\leq& C_f \left[ (1-\theta)\|\Delta_h{\tilde u}^{n}\|+\theta\|\Delta_h{\tilde u}^{n-1}\|\right]\\\label{D-Nn}
\leq& C_f \|{\tilde w}^{(n-\theta)}\| + C_f C_uh^2.
\end{align}
So \eqref{con-est1} and \eqref{D-Nn} lead to
\begin{align*}
\sum_{k=1}^nA_{n-k}^{(n)}\nabla_\tau(\|{\tilde w}^n\|^2+\nu^2\|\nabla_h{\tilde v}^n\|^2)
\leq 2\nu(1+\nu)C_f(1+C_\Omega) \left( \|{\tilde w}^{(n-\theta)}\|+\nu\|\nabla_h{\tilde v}^{(n-\theta)}\|\right)^2\\
+2\nu(1+\nu)(1+C_\Omega)\left({\cal T}^{n-\theta}+C_fC_uh^2\right) \left( \|{\tilde w}^{(n-\theta)}\|+\nu\|\nabla_h{\tilde v}^{(n-\theta)}\|\right).
\end{align*}
Applying Lemma \ref{gronwall-lemma}  and utilizing \eqref{PTtheta} again,  it yields
\begin{align*}
\|{\tilde w}^n\|+\nu\|\nabla_h{\tilde v}^n\|
\leq {\cal E}_n
\left[  2\nu(1+\nu)(1+C_\Omega)\max_{1\leq k\leq n}\sum_{j=1}^kP_{k-j}^{(k)} \left({\cal T}^{j-\theta}+C_fC_uh^2\right) \right]\leq {\cal E}_n {\tilde {\cal T}}^{n-\theta}.
\end{align*}
Therefore \eqref{induction-result} is verified.

Finally, the desired result \eqref{convergence-result} is reached by \eqref{error_eq2} and unifying the constants.
\end{proof}

\begin{remark}
A memory and computational storage saving technique (called SOE approximation)  investigated in \cite{FL1} (see also \cite[Theorem 2.5]{FL1} or \cite[Lemma 5.1]{Liao-AC2}) to compute the discrete Caputo derivative can be directly employed to the nonuniform L1 and Alikhanov formulas, the corresponding coefficients of fast L1 and fast Alikhanov formulas preserve the properties {\bf A1}--{\bf A2}  \cite{LiaoGronwall,Liao-AC2} which further ensure the theoretical analysis of the associated fast schemes. Therefore, in our later implementation of adaptive time stepping strategy and numerical tests, we will always utilize the fast L1 formula  \cite[Example 2]{LiaoGronwall} and the fast Alikhanov formula  \cite[eq. (5.3)]{Liao-AC2} while applying the proposed algorithms \eqref{sc1}--\eqref{sc3} with $\theta=0$ and $\theta=\beta/2$, respectively.
\end{remark}

\subsection{Adaptive time-stepping strategy}
The time mesh assumption in Theorems \ref{Convergence} permits us to establish adaptive time-stepping strategy based on the fast L1 and fast Alikhanov algorithms to reduce the computational costs while solving the semilinear diffusion-wave equation, especially when the solution of the governing problem may possess highly oscillatory feature in time. In the following, we refer to \cite{Gomez-JCP2011,Liao-AC2} for designing an adaptive time-stepping algorithm of the semilinear diffusion-wave equation \eqref{eq1}, the strategy is presented in Algorithm \ref{adaptive}.

\begin{algorithm}[htb]
\caption{Adaptive time-stepping strategy}\label{adaptive}
{\bf Given:}
$u^n$, $v^n$ and time step $\tau_{n+1}$
\begin{algorithmic}[1]
\STATE{Compute $u_1^{n+1}$ by the (fast) L1 scheme (\eqref{sc1}--\eqref{sc3} for $\theta=0$) with time step $\tau_{n+1}$;}
\STATE{Compute $u_2^{n+1}$ by the (fast) Alikhanov scheme (\eqref{sc1}--\eqref{sc3} for $\theta=\beta/2$) with time step $\tau_{n+1}$;}
\STATE{Calculate $e^{n+1}=\|u_2^{n+1}-u_1^{n+1}\|/\|u_2^{n+1}\|$;}
\IF{$e^{n+1}<tol$ or $\tau_{n+1}=\frac23\tau_{n}$}
\STATE{Update time step size $\tau_{n+2}\leftarrow \min\{\max\{\tau_{\min},\tau_{ada}\},\tau_{\max}\}$;}
\ELSE
 \STATE{Reformulate the time-step size $\tau_{n+1}\leftarrow \max\{\min\{\max\{\tau_{\min},\tau_{ada}\},\tau_{\max}\},\frac23\tau_{n+1}\}$;}
\STATE{\textbf{Goto} 1}
\ENDIF
\end{algorithmic}
\end{algorithm}

The adaptive time step size in Algorithm \ref{adaptive}  is updated by
$$\tau_{ada}(e,\tau)=S\left(\frac{tol}{e} \right)^{\frac12}\tau,$$
where $S$, $tol$ denote the safety coefficient and the tolerance, respectively.

\section{Numerical experiments}\label{Numerical}

Numerical examples are carried out in this section to show the accuracy and efficiency of proposed algorithms. The absolute tolerance error $\epsilon$ and the cut-off time $\Delta t$ of fast L1 formula  \cite[Example 2]{LiaoGronwall} and the fast Alikhanov formula  \cite[Lemma 5.1]{Liao-AC2} are set as $\epsilon=10^{-12}$ and $\Delta t = \tau_1$ in all of the following tests.

\begin{example}\label{ex1}
We first consider the problem \eqref{eq1} with $\Omega=(0,1)^2$, $T=1$, $\nu=1$ and
$$f(u,{\bf x},t)=-u^3+[\sin(\pi x)\sin(\pi y)(1+t+t^\alpha)]^3+\sin(\pi x)\sin(\pi y)\left[\Gamma(\alpha+1)+2\pi^2(1+t+t^\alpha)\right].$$
In this situation, the exact solution is $u=\sin(\pi x)\sin(\pi y)(1+t+t^\alpha)$.
\end{example}

One may notice that the regularity parameters in \eqref{regularity} are $\sigma_1=\alpha$ and $\sigma_2=\alpha/2$ for Example \ref{ex1}. Therefore, according to Theorem \ref{Convergence}, the optimal mesh parameter is $\gamma_{opt}=(4-\alpha)/\alpha$ for the nonuniform L1 scheme and takes the value $\gamma_{opt}=4/\alpha$ for the nonuniform Alikhanov scheme. They are all bounded for $\alpha\in(1,2)$. These bounded mesh parameters keep the robustness of the algorithms in practical implementation if the graded mesh $t_k=T(k/N)^\gamma$ is imposed to deal with the weak initial singularity. On the other hand, the optimal grading parameters of the nonuniform schemes in \cite{KoptevaMC2019,LiaoL1,LiaoSecondOrder,Liao-AC2,LiaoYanZhang2018,Stynes-SIAM2017}
will grow without bound while the fractional order becomes small as they are all possessing the form $\gamma_{opt}=r/\alpha$ where $\alpha\in(0,1)$ for the sub-diffusion problems and $r$ should be the optimal time rate, this generally lead to practical limitations. 

 Since the spatial error ${\cal O}(h^2)$ is standard, we only display the temporal accuracy of the fast L1 and fast Alikhanov schemes. For Example \ref{ex1}, we fixed a fine spatial grid mesh with $M=1000$  such that the temporal errors dominate the spatial errors. In each tests, the time interval $[0,T]$ is divided into two parts $[0,T_0]$ and $(T_0,T]$ with total $N$ time nodes. A graded mesh with $t_k=T_0(k/N_0)^\gamma$ in the first interval  $[0,T_0]$ is utilized to resolve the weak initial singularity, where $T_0=\min\{1/\gamma,T \}$. For the second interval $(T_0,T]$, we use  random time-step sizes $\tau_{N_0+k}=(T-T_0)\epsilon_k/\sum_{k=1}^{N_1}\epsilon_k$ for $N_1=N-N_0$, where $\epsilon_k$ take values in $(0,1)$ randomly. For this example, we take $N_0=\lceil \frac{N}{T+1-\gamma^{-1}}\rceil$. The discrete $H^2$-norm errors $e_{H^2}(N)=\max_{1\leq n\leq N}\|U^n-u^n\|_{H^2}$ are  recorded in each run, and the temporal convergence order is given by
         $$\mbox{Order}=\log_2\left[\frac{e_{H^2}(N/2)}{e_{H^2}(N)}\right].$$

\begin{table}[htb!]
 \begin{center}
 \caption{Numerical accuracy  in temporal direction of  fast L1 scheme for Example \ref{ex1}, where $\alpha=1.1$.}\label{table1}
 \renewcommand{\arraystretch}{1}
 \def\temptablewidth{0.95\textwidth}
 {\rule{\temptablewidth}{1pt}}
 \begin{tabular*}{\temptablewidth}{@{\extracolsep{\fill}}ccccccc}
&\multicolumn{2}{c}{$\gamma=1$}  &\multicolumn{2}{c}{$\gamma_{opt}=(4-\alpha)/\alpha\approx2.64$}  &\multicolumn{2}{c}{$\gamma=\frac98\gamma_{opt}\approx2.97$}\\
  \cline{2-3}\cline{4-5}\cline{6-7}
 $N$ &  $e_{H^2}(N)$ &Order & $e_{H^2}(N)$& Order & $e_{H^2}(N)$& Order\\\hline
 $16$   & 4.2668e-02  & $\ast$  & 1.4285e-01  & $\ast$  &2.4493e-01  & $\ast$\\
 $32$  & 3.3723e-02  & 0.34   & 3.5519e-02  & 2.01 & 5.1210e-02  & 2.26 \\
 $64$  & 2.2386e-02  & 0.59  & 1.0731e-02  & 1.73 & 1.2611e-02  & 2.02 \\
$128$ & 1.3688e-02  & 0.71  & 2.4621e-03  & 2.12 & 3.8861e-03  & 1.70 \\\hline
Theoretical Order && 0.55 && 1.45 && 1.45
\end{tabular*}
{\rule{\temptablewidth}{1pt}}
\end{center}
\end{table}

\begin{table}[htb!]
 \begin{center}
 \caption{Numerical accuracy  in temporal direction of fast L1 scheme for Example \ref{ex1}, where $\alpha=1.5$.}\label{table2}
 \renewcommand{\arraystretch}{1}
 \def\temptablewidth{0.95\textwidth}
 {\rule{\temptablewidth}{1pt}}
 \begin{tabular*}{\temptablewidth}{@{\extracolsep{\fill}}ccccccc}
&\multicolumn{2}{c}{$\gamma=1$}  &\multicolumn{2}{c}{$\gamma_{opt}=(4-\alpha)/\alpha\approx1.67$}  &\multicolumn{2}{c}{$\gamma=\frac98\gamma_{opt}\approx1.88$}\\
  \cline{2-3}\cline{4-5}\cline{6-7}
 $N$ &  $e_{H^2}(N)$ &Order & $e_{H^2}(N)$& Order & $e_{H^2}(N)$& Order\\\hline
 $16$  & 3.4875e-02  & $\ast$ & 2.2507e-01  & $\ast$    &2.5657e-01  & $\ast$\\
 $32$  & 1.2196e-02  & 1.52  & 3.6991e-02  & 2.61    & 3.4006e-02  & 2.92 \\
 $64$  & 8.7566e-03  & 0.48  & 1.2921e-02  & 1.52   & 1.3962e-02  & 1.28 \\
$128$  & 5.5637e-03  & 0.65  & 4.2362e-03  & 1.61   & 3.8805e-03  & 1.85 \\\hline
Theoretical Order && 0.75 && 1.25 && 1.25
\end{tabular*}
{\rule{\temptablewidth}{1pt}}
\end{center}
\end{table}

\begin{table}[htb!]
 \begin{center}
 \caption{Numerical accuracy  in temporal direction of fast L1 scheme for Example \ref{ex1}, where $\alpha=1.9$.}\label{table3}
 \renewcommand{\arraystretch}{1}
 \def\temptablewidth{0.95\textwidth}
 {\rule{\temptablewidth}{1pt}}
 \begin{tabular*}{\temptablewidth}{@{\extracolsep{\fill}}ccccccc}
&\multicolumn{2}{c}{$\gamma=1$}  &\multicolumn{2}{c}{$\gamma_{opt}=(4-\alpha)/\alpha\approx1.11$}  &\multicolumn{2}{c}{$\gamma=\frac98\gamma_{opt}\approx1.24$}\\
  \cline{2-3}\cline{4-5}\cline{6-7}
 $N$ &  $e_{H^2}(N)$ &Order & $e_{H^2}(N)$& Order & $e_{H^2}(N)$& Order\\\hline
 $16$   & 7.4641e-02  & $\ast$  & 7.4027e-02  & $\ast$  &1.3225e-01  & $\ast$\\
 $32$  & 3.7415e-02  & 1.00   & 3.4234e-02  & 1.11 & 3.4250e-02  & 1.95 \\
 $64$  & 1.7990e-02  & 1.06  & 1.6886e-02  & 1.02 & 1.6737e-02  & 1.03 \\
$128$  & 8.4156e-03  & 1.10  & 8.0910e-03 & 1.06 & 8.0567e-03  & 1.05 \\\hline
Theoretical Order && 0.95 &&1.05 && 1.05
\end{tabular*}
{\rule{\temptablewidth}{1pt}}
\end{center}
\end{table}

\begin{table}[htb!]
 \begin{center}
 \caption{Numerical accuracy  in temporal direction of fast  Alikhanov scheme for Example \ref{ex1}, where $\alpha=1.2$.}\label{table4}
 \renewcommand{\arraystretch}{1}
 \def\temptablewidth{0.95\textwidth}
 {\rule{\temptablewidth}{1pt}}
 \begin{tabular*}{\temptablewidth}{@{\extracolsep{\fill}}ccccccc}
&\multicolumn{2}{c}{$\gamma=1$}  &\multicolumn{2}{c}{$\gamma_{opt}=4/\alpha\approx3.33$}  &\multicolumn{2}{c}{$\gamma=\frac98\gamma_{opt}\approx3.75 $}\\
  \cline{2-3}\cline{4-5}\cline{6-7}
 $N$ &  $e_{H^2}(N)$ &Order & $e_{H^2}(N)$& Order & $e_{H^2}(N)$& Order\\\hline
 $16$  & 5.2656e-02  & $\ast$   & 1.2494e-01  & $\ast$    & .5496e-01  &$\ast$  \\
 $32$  & 3.2671e-02  & 0.69  & 3.3236e-02  & 1.91    & 4.1575e-02  & 1.90 \\
$64$  & 2.0683e-02  & 0.66   & 8.5962e-03  & 1.95    & 1.0801e-02  & 1.94 \\
$128$  & 1.1645e-02  & 0.83  &2.1990e-03  & 1.97    & 2.8352e-03  & 1.93\\\hline
Theoretical Order && 0.60 &&2.00 && 2.00
\end{tabular*}
{\rule{\temptablewidth}{1pt}}
\end{center}
\end{table}

\begin{table}[htb!]
 \begin{center}
 \caption{Numerical accuracy  in temporal direction of fast Alikhanov scheme for Example \ref{ex1}, where $\alpha=1.5$.}\label{table5}
 \renewcommand{\arraystretch}{1}
 \def\temptablewidth{0.95\textwidth}
 {\rule{\temptablewidth}{1pt}}
 \begin{tabular*}{\temptablewidth}{@{\extracolsep{\fill}}ccccccc}
&\multicolumn{2}{c}{$\gamma=1$}  &\multicolumn{2}{c}{$\gamma_{opt}=4/\alpha\approx2.67$}  &\multicolumn{2}{c}{$\gamma=\frac98\gamma_{opt}=3 $}\\
  \cline{2-3}\cline{4-5}\cline{6-7}
 $N$ &  $e_{H^2}(N)$ &Order & $e_{H^2}(N)$& Order & $e_{H^2}(N)$& Order\\\hline
 $16$  & 3.0823e-02  & $\ast$   & 7.0440e-02  & $\ast$    & 8.7416e-02  &$\ast$  \\
 $32$  & 1.3857e-02  & 1.15  & 1.8560e-02  & 1.92    & 2.3212e-02  & 1.91 \\
$64$  & 6.2024e-03  & 1.16   & 4.7736e-03  & 1.96    & 5.9919e-03  & 1.95 \\
$128$  & 2.6236e-03  & 1.24  &1.2150e-03  & 1.97    & 1.5269e-03  & 1.97\\\hline
Theoretical Order && 0.75 &&2.00 && 2.00
\end{tabular*}
{\rule{\temptablewidth}{1pt}}
\end{center}
\end{table}

\begin{table}[htb!]
 \begin{center}
 \caption{Numerical accuracy  in temporal direction of fast Alikhanov scheme for Example \ref{ex1}, where $\alpha=1.8$.}\label{table6}
 \renewcommand{\arraystretch}{1}
 \def\temptablewidth{0.95\textwidth}
 {\rule{\temptablewidth}{1pt}}
 \begin{tabular*}{\temptablewidth}{@{\extracolsep{\fill}}ccccccc}
&\multicolumn{2}{c}{$\gamma=1$}  &\multicolumn{2}{c}{$\gamma_{opt}=4/\alpha\approx2.22$}  &\multicolumn{2}{c}{$\gamma=\frac98\gamma_{opt}=2.50$}\\
  \cline{2-3}\cline{4-5}\cline{6-7}
 $N$ &  $e_{H^2}(N)$ &Order & $e_{H^2}(N)$& Order & $e_{H^2}(N)$& Order\\\hline
 $16$  & 1.9521e-02  & $\ast$   & 3.5560e-02  & $\ast$  & 4.3938e-02  & $\ast$ \\
 $32$  & 6.7203e-03  & 1.54  & 9.3089e-03  & 1.93  & 1.1590e-02  & 1.92 \\
$64$  & 2.6309e-03  & 1.35  & 2.3828e-03  & 1.97  &2.9755e-03 & 1.96 \\
$128$ & 1.1487e-03  & 1.20   &6.0470e-04  & 1.98  & 7.5559e-04 & 1.98\\\hline
Theoretical Order && 0.90 && 2.00 && 2.00
\end{tabular*}
{\rule{\temptablewidth}{1pt}}
\end{center}
\end{table}

Tables \ref{table1}--\ref{table3} record the numerical results of  the proposed fast L1 scheme with different  grading parameters when solving the example for different $\alpha$. One can observe that the L1 scheme works accurately with the optimal temporal convergence of ${\cal O}(\tau^{\min\{2-\frac{\alpha}{2},\gamma\frac{\alpha}{2}\}})$.
Similar numerical tests of the fast Alikhanov scheme are carried out for the example, and the results are listed in Tables \ref{table4}--\ref{table6}. The temporal convergence of ${\cal O}(\tau^{\min\{2,\gamma\frac{\alpha}{2}\}})$ is well reflected and the optimal second-order convergence is apparent while $\gamma\geq\gamma_{opt}={4}/{\alpha}$.

\begin{example}\label{ex2}
Consider the semilinear problem \eqref{eq1} with $\Omega=(-1,1)^2$, $\nu=1$ and
$f(u,{\bf x},t)=-u^3$. The initial data are given as
$$\varphi({\bf x})=(x^2-1)(y^2-1)\{\exp\{-10((x+0.4)^2+y^2)]+\exp[-10((x-0.4)^2+y^2)]\},\quad
{\tilde \varphi}({\bf x})=0.$$
\end{example}

\begin{figure}[htb!]
\centering
\includegraphics[scale=0.22]{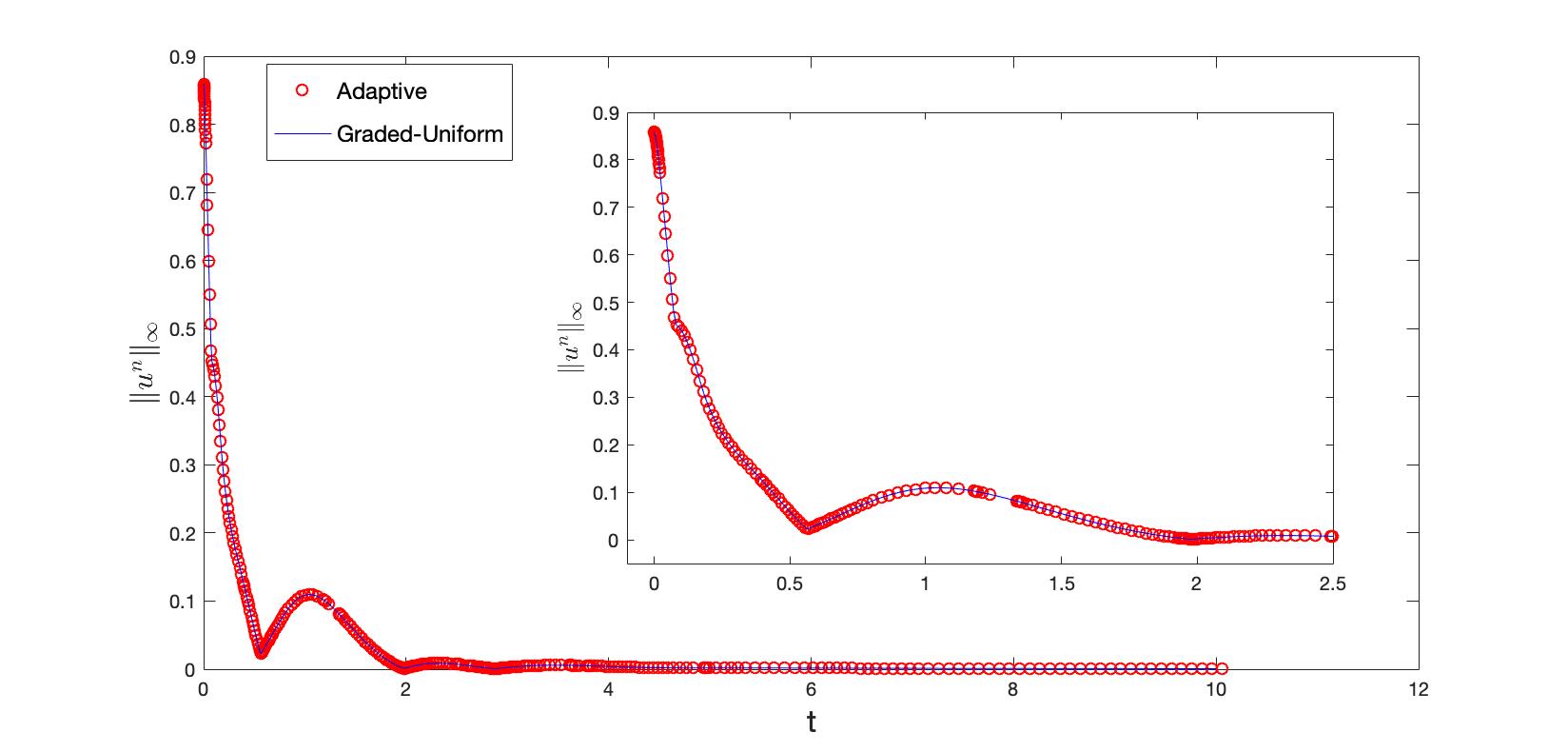}
\caption{The numerical solution in maximum-norm of the Algorithm \ref{adaptive} and the Graded-Uniform scheme for Example \ref{ex2} with $\alpha=1.5$.}
\label{fig1}
\end{figure}

\begin{figure}[htbp]
\subfigure[t=0]{
    \begin{minipage}[t]{0.4\linewidth}
        \center
       \includegraphics[width=1.8in]{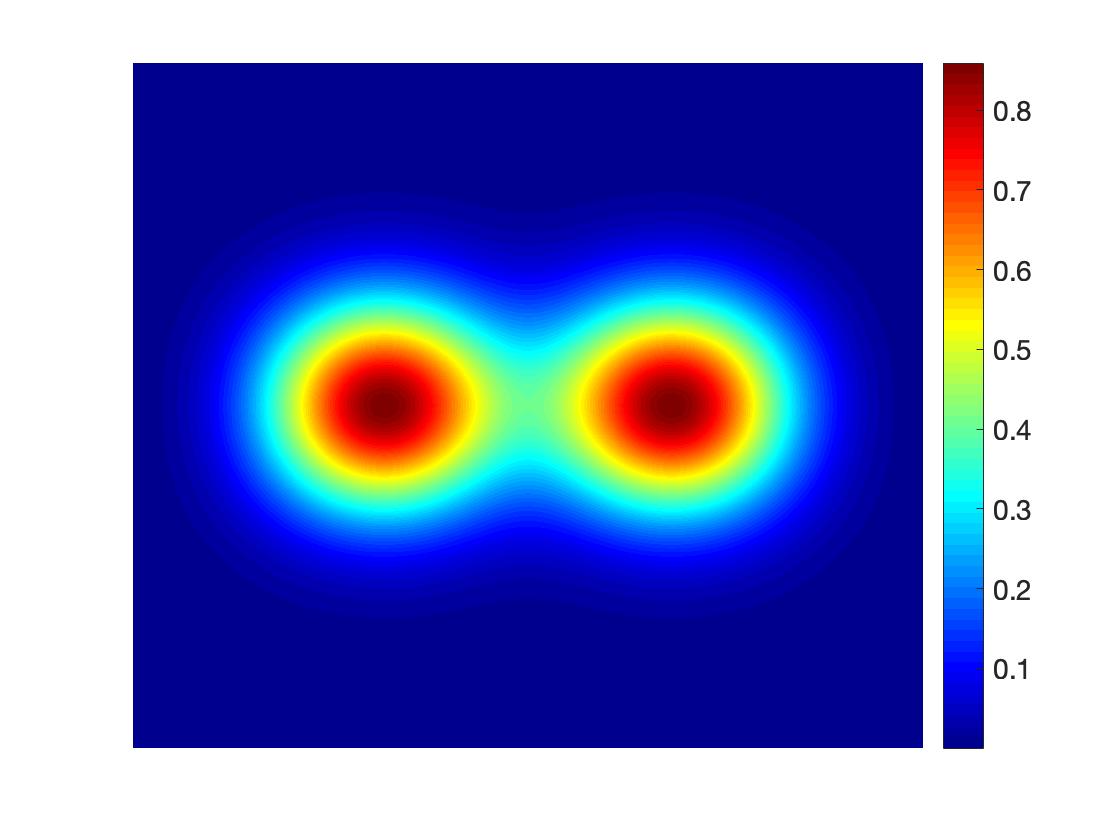}
    \end{minipage}
}
\subfigure[t=0.5]{
    \begin{minipage}[t]{0.4\linewidth}
       \center
       \includegraphics[width=1.8in]{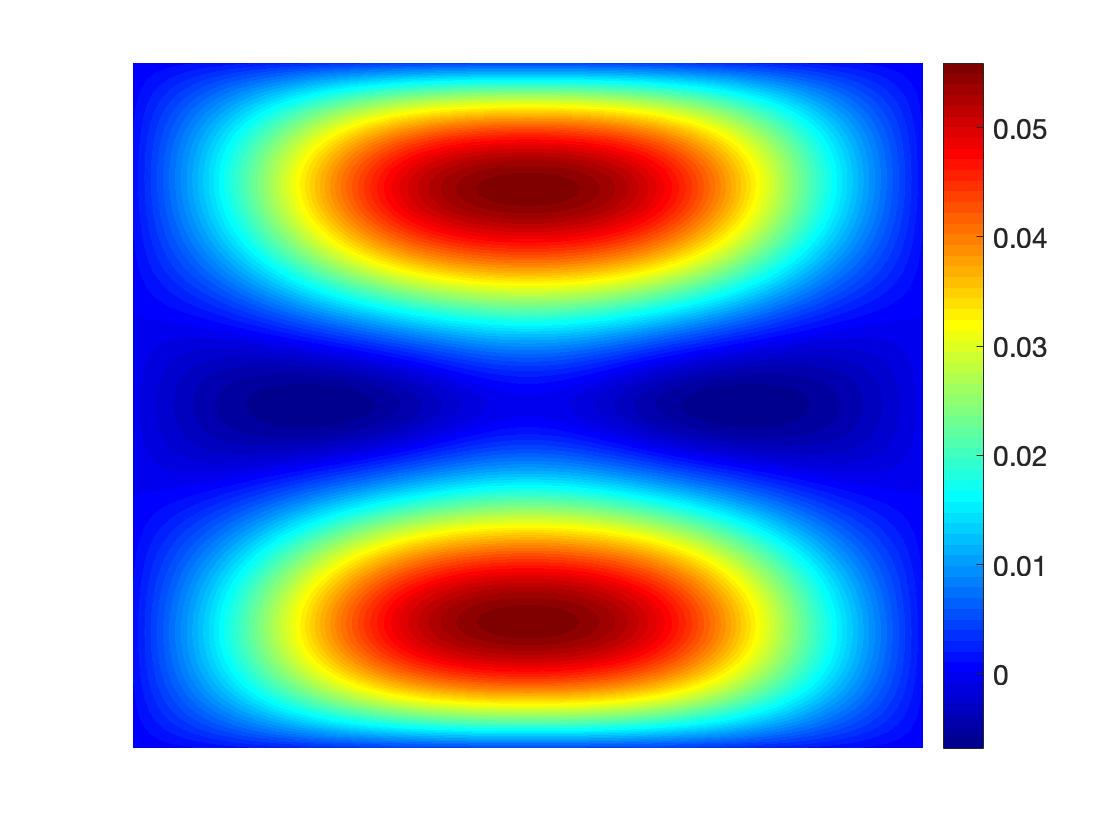}
    \end{minipage}
}
\subfigure[t=2]{
    \begin{minipage}[t]{0.4\linewidth}
        \center
       \includegraphics[width=1.8in]{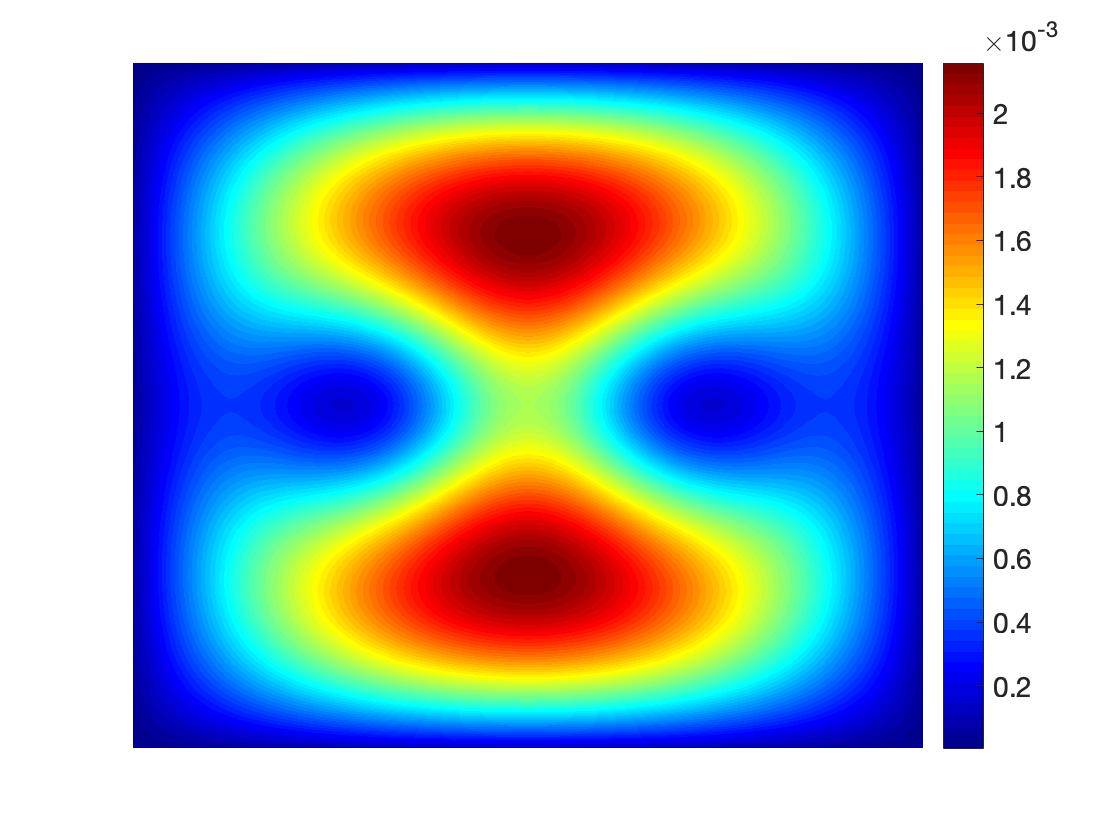}
    \end{minipage}
}
\subfigure[t=10]{
    \begin{minipage}[t]{0.4\linewidth}
        \center
       \includegraphics[width=1.8in]{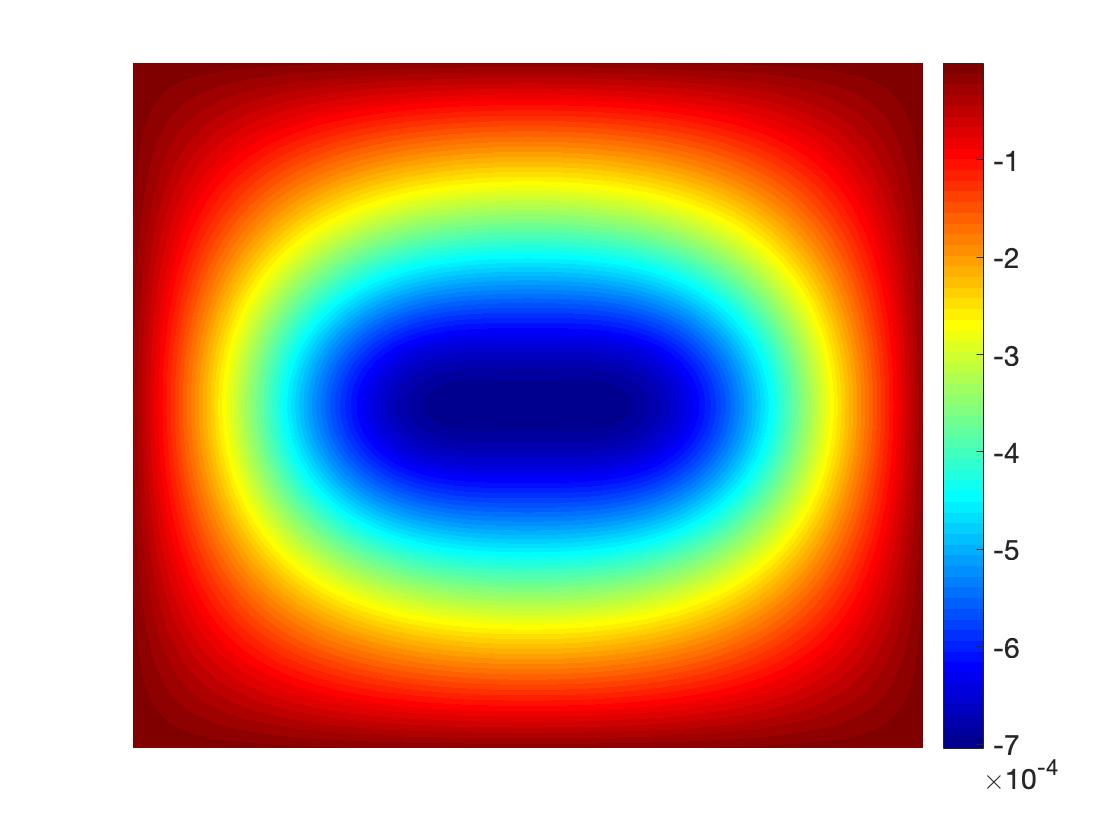}
    \end{minipage}
}
\centering
\caption{Contour plots of the solutions of Algorithm \ref{adaptive}  for Example \ref{ex2} at different time with $\alpha=1.5$.}
\label{fig2}
\end{figure}

\begin{figure}[htb!]
\centering
\includegraphics[scale=0.22]{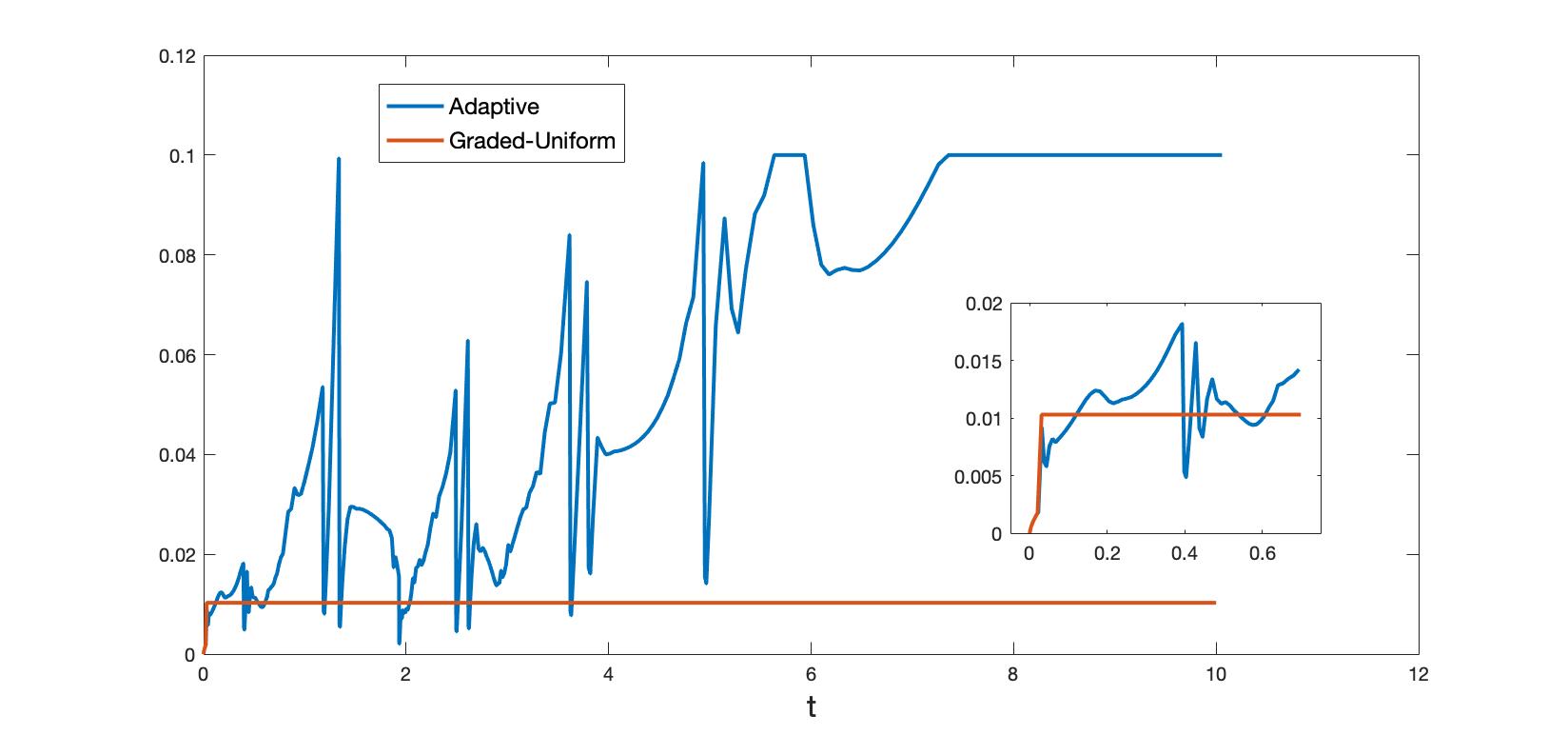}
\caption{The variation of time step sizes of the Algorithm \ref{adaptive} and the Graded-Uniform scheme for Example \ref{ex2} with $\alpha=1.5$.}
\label{fig3}
\end{figure}

For Example \ref{ex2}, we choose the spatial node $M=100$ and also divide the time interval $[0,T]$ into two parts $[0,T_0]$ and $(T_0,T]$ with $T_0=0.02$. The Alikhanov algorithm on graded mesh with $t_k=T_0(k/N_0)^\gamma$ ($\gamma=\alpha/4$) in the first interval  $[0,T_0]$ is utilized to resolve the possible weak initial singularity, where $N_0=30$. For the remaining interval $(T_0,T]$, we employ the proposed adaptive time stepping strategy (Algorithm \ref{adaptive}) to compute the numerical solution until $T=10$. The parameters of the adaptive algorithm for solving this example are
$$tol=10^{-3},~S=0.9,~\tau_{\min}=10^{-3},~\tau_{\max}=10^{-1},~\tau_{N_0+1}=\tau_{N_0}.$$

In order to show the efficiency of the adaptive algorithm, the fast linearized Alikhanov scheme is applied at the same time to find the solution in the interval $(T_0,T]$. Its temporal mesh is graded (with $\gamma=\alpha/4$) in $[0,T_0]$ and is uniform in $(T_0,T]$.
 In the following, we use `Graded-Uniform' to represent this scheme.

Figure \ref{fig1} displays the numerical solution in maximum-norm of the Algorithm \ref{adaptive} and the Graded-Uniform scheme for $\alpha=1.5$. It implies that the adaptive mesh suits well with a dense uniform mesh in $(T_0,T]$,  provided that the adaptive mesh requires 277 time nodes in the remain interval $(T_0,T]$ whereas the uniform mesh needs 970 time nodes. Figure \ref{fig2} gives the solution contour plots of the solutions by the adaptive strategy, which simply shows the wave interactions of the example at different time.  The variation of the temporal step sizes of the adaptive strategy with its comparison to those of the Graded-Uniform scheme are presented in Figure \ref{fig3}. The results indicate that the adaptive time-stepping strategy should be efficient and robust in the long time simulation of the semilinear diffusion-wave equations especially when the solution may exhibit high oscillations  in time.

\section{Concluding remarks}\label{Conclusion}

We proposed a novel order reduction method to equivalently rewrite the semilinear diffusion-wave equation into coupled equations, where the explicit time-fractional derivative orders are all $\alpha/2$. The L1 and Alikhanov schemes combining with linearized approximations have been constructed for the equivalent problem. By using $H^2$ energy method, unconditional convergences (Theorem \ref{Convergence}) were obtained for the two proposed algorithms under reasonable regularity assumptions and weak mesh restrictions. An adaptive time-stepping strategy was then designed for the semilinear problem to deal with possible temporal oscillations of the solution.  The theoretical results were well demonstrated by our numerical experiments.

We finally point out several relevant issues that deserve for further study: (i) deriving the regularity of the linear and semilinear diffusion-wave equations for the difference schemes;  (ii) studying the energy properties of the nonlinear diffusion-wave equations for both the continuous and discrete versions, noting that corresponding properties were investigated recently for the nonlinear sub-diffusion problems \cite{TangYuZhouSISC}; (iii) extending the proposed methods to some related problems, such as the multi-term time-fractional wave equation \cite{LyuLiangWangANM2020}.

\section{Acknowledgement}\label{Acknowledgement}
The authors are very grateful to Prof. Hong-lin Liao for his great help on the design of the SFOR method and valuable suggestions on other parts of the whole paper.

\section{Appendix: Truncation error analysis}\label{Appendix}

The truncation errors in \eqref{error_eq1}--\eqref{error_eq3} are defined as
\begin{align*}
&({\cal T}_f)_h^{n-\theta}:=f(U_h^{n-\theta},{\bf x}_h,t_{n-\theta})-\left[ f(U_h^{n-1},{\bf x}_h,t_{n-\theta})+(1-\theta)f'(U_h^{n-1},{\bf x}_h,t_{n-\theta})\nabla_\tau U_h^{n} \right],\\
&({\cal T}_{v1})_h^{n-\theta}:={\cal D}_t^\beta {\bf v}({\bf x}_h,t_{n-\theta})-( {\cal D}_\tau^\beta {\bf v}_h)^{n-\theta},\\
&({\cal T}_w)_h^{n-\theta}:=w({\bf x}_h,t_{n-\theta})-w_h^{n-\theta},\\
&({\cal T}_{u})_h^{n-\theta}:={\cal D}_t^\beta {\bf u}({\bf x}_h,t_{n-\theta})-( {\cal D}_\tau^\beta {\bf u}_h)^{n-\theta},\\
&({\cal T}_{v2})_h^{n-\theta}:={\bf v}({\bf x}_h,t_{n-\theta})-{\bf v}_h^{n-\theta},\\
&{\cal S}_h^{n}:=\Delta u({\bf x}_h,t_n)-\Delta_h u_h^n.
\end{align*}
According to \cite[Lemma 3.8 and Theorem 3.9]{LiaoSecondOrder}, we have the follow lemma on estimating the time weighted approximation.
\begin{lemma}\label{theta-error}
Assume that $g\in C^2((0,T])$ and there exists a constant $C_g>0$ such that
$$|g''(t)|\leq C_g(1+t^{\sigma-2}), \quad 0<t\leq T,$$
where $\sigma\in(0,1)\cup (1,2)$ is a regularity parameter.
Denote the local truncation error of $g^{n-\vartheta}$ (here $\vartheta=\beta/2$) by
$${\tilde{\cal R}}^{n-\vartheta}=g(t_{n-\vartheta})-g^{n-\vartheta}, \quad 1\leq n\leq N.$$
If the mesh assumption {\bf MA} holds, then
\begin{align*}
\sum_{j=1}^n P_{n-j}^{(n)}|{\tilde{\cal R}}^{j-\vartheta}|\leq C_g\left(\tau_1^{\sigma+\beta}/\sigma+t_n^\beta\max_{2\leq k\leq n} t_{k-1}^{\sigma-2}\tau_k^2 \right)\leq C\tau^{\min\{\gamma\sigma,2\}}.
\end{align*}
\end{lemma}

The following lemma is provided to analyze $({\cal T}_f)_h^{n-\theta}$, which  is analogous  to Lemma 3.4 in \cite{LiaoYanZhang2018}.
\begin{lemma}\label{Rf-error}
Assume that $\eta\in C([0,T])\cap C^2((0,T])$ and there exists a constant $C_u>0$ such that
$$|\eta^{(k)}(t)|\leq C_u(1+t^{\sigma-k}), \quad 0<t\leq T,\quad k=1,2,$$
where $\sigma\in(0,1)\cup (1,2)$ is a regularity parameter. Assume further that  the nonlinear function $f(u,x,t)\in C^4({\mathbb R})$ with respect to $u$.
Denote $\eta^n=\eta(t_n)$ and the local truncation error
\begin{align*}
{\cal R}_f^{n-\theta}:=f(\eta(t_{n-\theta}),x,t)-\left[f(\eta^{n-1},x,t)+(1-\theta)f'_u(\eta^{n-1},x,t)\nabla_\tau \eta^n\right].
\end{align*}
If the assumption {\bf MA} holds, then
\begin{equation}\nonumber
\sum_{j=1}^n P_{n-j}^{(n)}|{\cal R}_f^{j-\theta}|\leq \left\{
\begin{array}{ll}
C\tau^{\min\{2\gamma\sigma,2\}},\quad \theta=0;\\
C\tau^{\min\{\gamma\sigma,2\}},\quad \theta=\frac{\beta}{2};
\end{array}
 \quad 1\leq n\leq N.
 \right.
\end{equation}
\end{lemma}
\begin{proof}
Denote ${\cal R}_\eta^{n-\theta}:=\eta(t_{n-\theta})-\eta^{n-\theta}$. We have ${\cal R}_\eta^{n-\theta}=0$ while $\theta=0$.
By the Taylor expansion, we have
\begin{align*}
{\cal R}_f^{n-\theta}
=&f_u'(\eta^{n-1}){\cal R}_\eta^{n-\theta}\\
&+\left((1-\theta)\nabla_\tau \eta^n+{\cal R}_\eta^{n-\theta}\right)^2\int_0^1f''_u\left( \eta^{n-1}+s (\eta(t_{n-\theta})-\eta^{n-1}),x,t\right)(1-s)\zd s.
\end{align*}
Following the proof of \cite[Lemma 3.4]{LiaoYanZhang2018} and using Lemma \ref{theta-error}, the desired result holds immediately.
\end{proof}
For ${\bf x}\in \Omega$, let $\xi^n({\bf x})$  be a
spatially
continues function and denote $\xi^n_h:=\xi^n({\bf x}_h)$. One may apply the Taylor expansion to get
\begin{align*}
\Delta_h \xi^n_h=&\int_0^1\left[\partial_{xx}\xi^n(x_i-sh_x,y_j)+\partial_{xx}\xi^n(x_i+sh_x,y_j) \right](1-s)\zd s\\
&+\int_0^1\left[\partial_{yy}\xi^n(x_i,y_j-sh_y)+\partial_{yy}\xi^n(x_i,y_j+sh_y) \right](1-s)\zd s, \quad 1\leq n\leq N.
\end{align*}
Then we  define a function ${\cal T}_f^{n}({\bf x})$ by $({\cal T}_f)_h^{n}={\cal T}_f^{n}({\bf x}_h)$. If the assumptions in \eqref{regularity} and {\bf MA} are satisfied, by Lemma \ref{Rf-error} and  the differential formula of composite function, we can obtain
\begin{equation}\label{PTf}
\sum_{j=1}^n P_{n-j}^{(n)}\|\Delta_h({\cal T}_f)^{n-\theta}\|\leq \left\{
\begin{array}{ll}
C\tau^{\min\{2\gamma\sigma_1,2\}},\quad \theta=0;\\
C\tau^{\min\{\gamma\sigma_1,2\}},\quad \theta=\frac{\beta}{2};
\end{array}
 \quad 1\leq n\leq N.
 \right.
 \end{equation}
For the spatial error, based on the regularity condition, it is easy to know that
\begin{align}\label{Sn}
\|{\cal S}^{n}\|\leq C_uh^2, \quad 1\leq n\leq N.
\end{align}
Then
\begin{align}\label{PDSn}
\sum_{j=1}^n P_{n-j}^{(n)} \|({\cal D}_\tau^{\beta}{\cal S})^{j}\|\leq
\sum_{j=1}^n P_{n-j}^{(n)}\sum_{k=1}^jA_{j-k}^{(j)}\|\nabla_\tau{\cal S}^{k}\|
=\sum_{k=1}^n\|\nabla_\tau{\cal S}^{k}\|\leq C_u(1+t_{n}^{\sigma_1-1})h^2.
\end{align}
We now consider the temporal truncation errors $({\cal T}_{v1})_h^{n-\theta}$, $({\cal T}_{w})_h^{n-\theta}$, $({\cal T}_{u})_h^{n-\theta}$ and $({\cal T}_{v2})_h^{n-\theta}$ in two situations: $\theta=0$ and $\theta=\beta/2$.

For a function $g(t)$, define the global error
$${\cal R}^{n-\theta}:=({\cal D}_t^\beta g)(t_{n-\theta})-({\cal D}_{\tau}^\beta g)^{n-\theta},\quad 1\leq n\leq N.$$
$\bullet$ {\bf For L1 approximation ($\theta=0$):} We have $({\cal T}_w)_h^{n}=({\cal T}_{v2})_h^{n}=0$ in this situation.

According to \cite[Lemma 3.3]{LiaoL1} and \cite[Lemma 3.3]{LiaoYanZhang2018}, the global  consistency error of the L1 approximation can be presented in the following lemma.
\begin{lemma}\label{L1-error}
Assume that $g\in C^2((0,T])$ and there exists a constant $C_g>0$ such that
$$|g''(t)|\leq C_g(1+t^{\sigma-2}), \quad 0<t\leq T,$$
where $\sigma\in(0,1)\cup (1,2)$ is a regularity parameter. If the assumption {\bf MA} holds, it follows that
\begin{align*}
\sum_{j=1}^n P_{n-j}^{(n)}|{\cal R}^j|\leq C_g\left(\tau_1^\sigma/\sigma+\frac1{1-\beta}\max_{2\leq k\leq n}(t_k-t_1)^\beta t_{k-1}^{\sigma-2}\tau_k^{2-\beta} \right)\leq C\tau^{\min\{2-\beta,\gamma\sigma\}}.
\end{align*}
\end{lemma}

Define the functions ${\cal T}_{v1}^{n}({\bf x})$ and ${\cal T}_u^{n}({\bf x})$ by $({\cal T}_{v1})_h^n:={\cal T}_{v1}^{n}({\bf x}_h)$ and $({\cal T}_{u})_h^n:={\cal T}_u^{n}({\bf x}_h)$ respectively.
Using similar techniques for \eqref{PTf}, and  Lemma \ref{L1-error} with the in assumptions \eqref{regularity} and {\bf MA}, we have
\begin{align}\label{PTv10-u0}
\sum_{j=1}^n P_{n-j}^{(n)}\|\Delta_h{\cal T}_{v1}^{n}\| \leq C\tau^{\min\{2-\beta,\gamma\sigma_2\}}
\quad\mbox{and} \quad
\sum_{j=1}^n P_{n-j}^{(n)}\|\Delta_h{\cal T}_{u}^{n}\| \leq C\tau^{\min\{2-\beta,\gamma\sigma_1\}}.
\end{align}
$\bullet$ {\bf For Alikhanov approximation ($\theta=\beta/2$):}

The global consistency error estimate of the Alikhanov approximation is estimated in the next lemma.
\begin{lemma}(\cite[Lemma 3.6]{LiaoSecondOrder})\label{Alikhanov-error}
Assume that $g\in C^3((0,T])$ and there exists a constant $C_g>0$ such that
$$|g'''(t)|\leq C_g(1+t^{\sigma-3}), \quad 0<t\leq T,$$
where $\sigma\in(0,1)\cup (1,2)$ is a regularity parameter. Then
\begin{align*}
\sum_{j=1}^n P_{n-j}^{(n)}|{\cal R}^{j-\theta}|\leq C_g\left(\tau_1^\sigma/\sigma+t_1^{\sigma-3}\tau_2^3+\frac1{1-\beta}\max_{2\leq k\leq n}t_k^\beta t_{k-1}^{\sigma-3}\tau_k^3/\tau_{k-1}^\beta \right).
\end{align*}
\end{lemma}

By Lemma \ref{Alikhanov-error}, Lemma \ref{theta-error}, the assumptions in \eqref{regularity} and {\bf MA}, it is easy to get that
\begin{align}\label{PTv1}
&\sum_{j=1}^nP_{n-j}^{(n)} \|\Delta_h({\cal T}_{v1})_h^{j-\theta}\|
\leq C\left(\tau_1^{\sigma_2}+\tau_2^3\tau_1^{\sigma_2-3}+\max_{2\leq k\leq n} (t_k-t_1)^{\beta}t_{k-1}^{\sigma_2-3}\tau_k^{3-\beta} \right)\leq C\tau^{\min\{3-\beta,\gamma\sigma_2\}},\\\label{PTu}
&\sum_{j=1}^nP_{n-j}^{(n)} \|\Delta_h({\cal T}_u)_h^{j-\theta}\|
\leq C\left(\tau_1^{\sigma_1}+\tau_2^3\tau_1^{\sigma_1-3}+\max_{2\leq k\leq n} (t_k-t_1)^{\beta}t_{k-1}^{\sigma_1-3}\tau_k^{3-\beta} \right)\leq C\tau^{\min\{3-\beta,\gamma\sigma_1\}},\\\label{PTw}
&\sum_{j=1}^nP_{n-j}^{(n)} \|\Delta_h({\cal T}_w)_h^{j-\theta}\|
\leq C\left(\tau_1^{\sigma_1+\beta}+\max_{2\leq k\leq n}  t_{k-1}^{\sigma_1-2}\tau_k^2 \right)\leq C\tau^{\min\{2,\gamma\sigma_1\}},\\\label{PTv2}
&\sum_{j=1}^nP_{n-j}^{(n)} \|\Delta_h({\cal T}_{v2})_h^{j-\theta}\|
\leq C\left(\tau_1^{\sigma_2+\beta}+\max_{2\leq k\leq n}  t_{k-1}^{\sigma_2-2}\tau_k^2 \right)\leq C\tau^{\min\{2,\gamma\sigma_2\}}.
\end{align}

\bibliographystyle{siamplain}
\bibliography{references}

\begin{thebibliography}{10}

\bibitem{AAA}
{\sc A.~A. Alikhanov}, {\em A new difference scheme for the time fractional
  diffusion equation}, J. Comput, Phys., 280 (2015), pp.~424--438,
  \url{https://doi.org/10.1016/j.jcp.2014.09.031}.

\bibitem{ChenStynesJSC2019}
{\sc H.~Chen and M.~Stynes}, {\em Error analysis of a second-order method on
  fitted meshes for a time-fractional diffusion problem}, J. Sci. Comput., 79
  (2019), pp.~624--647, \url{https://doi.org/10.1007/s10915-018-0863-y}.

\bibitem{Cuesta2012SP}
{\sc E.~Cuesta, M.~Kirane, and S.~A. Malik}, {\em Image structure preserving
  denoising using generalized fractional time integrals}, Signal Processing, 92
  (2012), pp.~553--563, \url{https://doi.org/10.1016/j.sigpro.2011.09.001}.

\bibitem{Cuesta2006MC}
{\sc E.~Cuesta, C.~Lubich, and C.~Palencia}, {\em Convolution quadrature time
  discretization of fractional diffusion-wave equations}, Math. Comput., 75
  (2006), pp.~673--696, \url{https://doi.org/10.1090/S0025-5718-06-01788-1}.

\bibitem{Cuesta2003ANM}
{\sc E.~Cuesta and C.~Palencia}, {\em A fractional trapezoidal rule for
  integro-differential equations of fractional order in banach spaces}, Appl.
  Numer. Math., 45 (2003), pp.~139--159,
  \url{https://doi.org/10.1016/S0168-9274(02)00186-1}.

\bibitem{Cuesta2003SIAM}
{\sc E.~Cuesta and C.~Palencia}, {\em A numerical method for an
  integro-differential equation with memory in {B}anach spaces: qualitative
  properties}, SIAM J. Numer. Anal., 41 (2003), pp.~1232--1241,
  \url{https://doi.org/10.1137/S0036142902402481}.

\bibitem{Gomez-JCP2011}
{\sc H.~Gomez and T.~J.~R. Hughes}, {\em Provably unconditionally stable,
  second-order time-accurate, mixed variational methods for phase-field
  models}, J. Comput. Phys., 230 (2011), pp.~5310--5327,
  \url{https://doi.org/10.1016/j.jcp.2011.03.033}.

\bibitem{FL1}
{\sc S.~Jiang, J.~Zhang, Q.~Zhang, and Z.~Zhang}, {\em Fast evaluation of the
  {C}aputo fractional derivative and its applications to fractional diffusion
  equations}, Commun. Comput. Phys., 21 (2017), pp.~650--678,
  \url{https://doi.org/10.4208/cicp.OA-2016-0136}.

\bibitem{JinIMA}
{\sc B.~Jin, R.~Lazarov, and Z.~Zhou}, {\em An analysis of the {L}1 scheme for
  the subdiffusion equation with nonsmooth data}, IMA J. Numer. Anal., 36
  (2016), pp.~197--221, \url{https://doi.org/10.1093/imanum/dru063}.

\bibitem{JinSIAMJSC2016}
{\sc B.~Jin, R.~Lazarov, and Z.~Zhou}, {\em Two fully discrete schemes for
  fractional diffusion and diffusion-wave equations with nonsmooth data}, SIAM
  J. Sci. Comput., 38 (2016), pp.~A146--A170,
  \url{https://doi.org/10.1137/140979563}.

\bibitem{Jin-NM2018}
{\sc B.~Jin, B.~Li, and Z.~Zhou}, {\em Discrete maximal regularity of
  time-stepping schemes for fractional evolution equations}, Numer. Math., 138
  (2018), pp.~101--131, \url{https://doi.org/10.1007/s00211-017-0904-8}.

\bibitem{KoptevaMC2019}
{\sc N.~Kopteva}, {\em Error analysis of the {L}1 method on graded and uniform
  meshes for a fractional-derivative problem in two and three dimensions},
  Math. Comput., 88 (2019), pp.~2135--2155,
  \url{https://doi.org/10.1090/mcom/3410}.

\bibitem{LiWangXie2019}
{\sc B.~Li, T.~Wang, and X.~Xie}, {\em Analysis of the {L}1 scheme for
  fractional wave equations with nonsmooth data}, arXiv: 1908.09145v2
  [math.NA].

\bibitem{LiWangXie2019JSC}
{\sc B.~Li, T.~Wang, and X.~Xie}, {\em Analysis of a time-stepping
  discontinuous {G}alerkin method for fractional diffusion-wave equations with
  nonsmooth data}, J. Sci. Comput., 82 (2020),
  \url{https://doi.org/10.1007/s10915-019-01118-7}.

\bibitem{LiaoL1}
{\sc H.~L. Liao, D.~Li, and J.~Zhang}, {\em Sharp error estimate of a
  nonuniform {L}1 formula for time-fractional reaction-subdiffusion equations},
  SIAM J. Numer. Anal., 56 (2018), pp.~1112--1133,
  \url{https://doi.org/10.1137/17M1131829}.

\bibitem{LiaoSecondOrder}
{\sc H.~L. Liao, W.~McLean, and J.~Zhang}, {\em A second-order scheme with
  nonuniform time steps for a linear reaction-subdiffusion problem},
  arXiv:1803.09873v2 [math.NA].

\bibitem{LiaoGronwall}
{\sc H.~L. Liao, W.~McLean, and J.~Zhang}, {\em A discrete {G}r\"{o}nwall
  inequality with applications to numerical schemes for subdiffusion problems},
  SIAM J. Numer. Anal., 57 (2019), pp.~218--237,
  \url{https://doi.org/10.1137/16M1175742}.

\bibitem{Liao-NMPDE2010}
{\sc H.~L. Liao and Z.~Z. Sun}, {\em Maximum norm error bounds of {ADI} and
  compact {ADI} methods for solving parabolic equations}, Numer. Meth. Part
  Differ. Equ., 26 (2010), pp.~37--60, \url{https://doi.org/10.1002/num.20414}.

\bibitem{Liao-AC2}
{\sc H.~L. Liao, T.~Tang, and T.~Zhou}, {\em A second-order and nonuniform
  time-stepping maximum-principle preserving scheme for time-fractional
  {Allen-Cahn} equations}, J. Comput. Phys., 414 (2020),
  \url{https://doi.org/10.1016/j.jcp.2020.109473}.

\bibitem{LiaoYanZhang2018}
{\sc H.~L. Liao, Y.~Yan, and J.~Zhang}, {\em Unconditional convergence of a
  two-level linearized fast algorithm for semilinear subdiffusion equations},
  J. Sci. Comput., 80 (2019), pp.~1--25,
  \url{https://doi.org/10.1007/s10915-019-00927-0}.

\bibitem{LinXu2007}
{\sc Y.~Lin and C.~Xu}, {\em Finite difference/spectral approximations for the
  time-fractional diffusion equation}, J. Comput. Phys., 225 (2007),
  pp.~1533--1552, \url{https://doi.org/10.1016/j.jcp.2007.02.001}.

\bibitem{LL-Xie}
{\sc H.~Luo, B.~Li, and X.~Xie}, {\em Convergence analysis of a
  {Petrov-Galerkin} method for fractional wave problems with nonsmooth data},
  J. Sci. Comput., 80 (2019), pp.~957--992,
  \url{https://doi.org/10.1007/s10915-019-00962-x}.

\bibitem{LyuLiangWangANM2020}
{\sc P.~Lyu, Y.~Liang, and Z.~Wang}, {\em A fast linearized finite difference
  method for the nonlinear multi-term time-fractional wave equation}, Appl.
  Numer. Math., 151 (2020), pp.~448--471,
  \url{https://doi.org/10.1016/j.apnum.2019.11.012}.

\bibitem{Mainardi2010}
{\sc F.~Mainardi}, {\em Fractional Calculus and Waves in Linear
  Viscoelasticity}, Imperial College Press, London, 2010.

\bibitem{Mainardi2001}
{\sc F.~Mainardi and P.~Paradisi}, {\em Fractional diffusive waves}, J. Comput.
  Acoust., 9 (2001), pp.~1417--1436,
  \url{https://doi.org/10.1142/S0218396X01000826}.

\bibitem{MustaphaNM2007}
{\sc W.~McLean and K.~Mustapha}, {\em A second-order accurate numerical method
  for a fractional wave equation}, Numer. Math., 105 (2007), pp.~481--510,
  \url{https://doi.org/10.1007/s00211-006-0045-y}.

\bibitem{McLeanThomee2004}
{\sc W.~McLean and V.~Thom$\acute{\mbox{e}}$e}, {\em Time discretization of an
  evolution equation via {L}aplace transforms}, IMA J. Numer. Anal., 24 (2004),
  pp.~439--463, \url{https://doi.org/10.1093/imanum/24.3.439}.

\bibitem{McLeanThomee2010}
{\sc W.~McLean and V.~Thom$\acute{\mbox{e}}$e}, {\em Maximum-norm error
  analysis of a numerical solution via {L}aplace transformation and quadrature
  of a fractional-order evolution equation}, IMA J. Numer. Anal., 30 (2010),
  pp.~208--230, \url{https://doi.org/10.1093/imanum/drp004}.

\bibitem{MustaphaSIAM2013}
{\sc K.~Mustapha and W.~McLean}, {\em Superconvergence of a discontinuous
  {G}alerkin method for fractional diffusion and wave equations}, SIAM J. Nmer.
  Anal., 51 (2013), pp.~491--515, \url{https://doi.org/10.1137/120880719}.

\bibitem{MustaphaIMA2014}
{\sc K.~Mustapha and D.~Sch\"{o}tzau}, {\em Well-posedness of hp-version
  discontinuous {G}alerkin methods for fractional diffusion wave equations},
  IMA J. Numer. Anal., 34 (2014), pp.~1426--1446,
  \url{https://doi.org/10.1093/imanum/drt048}.

\bibitem{Nigmatullin}
{\sc R.~R. Nigmatullin}, {\em To the theoretical explanation of the
  ``{Universal Response}"}, Phys. Status Solidi B, 123 (1984), pp.~739--745,
  \url{https://doi.org/10.1002/pssb.2221230241}.

\bibitem{Oldham1974}
{\sc K.~Oldham and J.~Spanier}, {\em The Fractional Calculus}, Academic Press,
  New York, London, 1974.

\bibitem{Podnubny}
{\sc I.~Podnubny}, {\em Fractional Differential Equations}, Academic Press, San
  Diego, London, 1999.

\bibitem{Saka-Yama-moto}
{\sc K.~Sakamoto and M.~Yamamoto}, {\em Initial value/boundary value problems
  for fractional diffusion-wave equations and applications to some inverse
  problems}, J. Math. Anal. Appl., 382 (2011), pp.~426--447,
  \url{https://doi.org/10.1016/j.jmaa.2011.04.058}.

\bibitem{Stynes-SIAM2017}
{\sc M.~Stynes, E.~O'Riordan, and J.~L. Gracia}, {\em Error analysis of a
  finite difference method on graded meshes for a time-fractional diffusion
  equation}, SIAM J. Numer. Anal., 55 (2017), pp.~1057--1079,
  \url{https://doi.org/10.1137/16M1082329}.

\bibitem{SunWu2006}
{\sc Z.~Z. Sun and X.~N. Wu}, {\em A fully discrete difference scheme for a
  diffusion-wave system}, Appl. Numer. Math., 56 (2006), pp.~193--209,
  \url{https://doi.org/10.1016/j.apnum.2005.03.003}.

\bibitem{TangYuZhouSISC}
{\sc T.~Tang, H.~Yu, and T.~Zhou}, {\em On energy dissipation theory and
  numerical stability for time-fractional phase-field equations}, SIAM J. Sci.
  Comput., 41 (2019), pp.~A3757--A3778,
  \url{https://doi.org/10.1137/18M1203560}.

\bibitem{YanSIAM2018}
{\sc Y.~Yan, M.~Khan, and N.~J. Ford}, {\em An analysis of the modified {L}1
  scheme for time-fractional partial differential equations with nonsmooth
  data}, SIAM J. Numer. Anal., 56 (2018), pp.~210--227,
  \url{https://doi.org/10.1137/16M1094257}.

\bibitem{ZhangW2014}
{\sc W.~Zhang, J.~Li, and Y.~Yang}, {\em A fractional diffusion-wave equation
  with non-local regularization for image denoising}, Signal Processing, 103
  (2014), pp.~6--15, \url{https://doi.org/10.1016/j.sigpro.2013.10.028}.

\bibitem{Zorich}
{\sc V.~A. Zorich}, {\em Mathematical Analysis I}, Springer, Berlin, 2004.

\end{thebibliography}

\end{document}